\documentclass[10pt]{amsart}
\usepackage{amscd}
\usepackage{amsmath}
\usepackage{amsfonts}
\usepackage{amssymb}
\usepackage{amsthm}
\usepackage{enumerate}
\usepackage[T1]{fontenc}
\usepackage{hyperref}
\usepackage{mathrsfs}
\usepackage{stackrel}
\usepackage[all]{xy}
\usepackage{yhmath}

\DeclareMathOperator{\gr}{gr}
\DeclareMathOperator{\Hom}{Hom}
\DeclareMathOperator{\End}{End}
\DeclareMathOperator{\Tor}{Tor}
\DeclareMathOperator{\Der}{Der}

\DeclareMathOperator{\Sp}{Sp}
\DeclareMathOperator{\Spf}{Spf}

\DeclareMathOperator{\Sym}{Sym}
\DeclareMathOperator{\aug}{aug}
\DeclareMathOperator{\rig}{rig}

\DeclareMathOperator{\an}{an}
\DeclareMathOperator{\GL}{GL}
\DeclareMathOperator{\Coh}{Coh}
\DeclareMathOperator{\Loc}{Loc}
\DeclareMathOperator{\coker}{coker}
\DeclareMathOperator{\im}{Im}

\DeclareMathOperator{\id}{id}

\DeclareMathOperator{\op}{op}

\begin{document}
\newtheorem{MainThm}{Theorem}
\renewcommand{\theMainThm}{\Alph{MainThm}}
\theoremstyle{definition}
\newtheorem*{examps}{Examples}
\newtheorem*{defn}{Definition}
\theoremstyle{plain}
\newtheorem*{lem}{Lemma}
\newtheorem*{prop}{Proposition}
\newtheorem*{rmk}{Remark}
\newtheorem*{rmks}{Remarks}
\newtheorem*{thm}{Theorem}
\newtheorem*{example}{Example}
\newtheorem*{examples}{Examples}
\newtheorem*{cor}{Corollary}
\newtheorem*{conj}{Conjecture}
\newtheorem*{hyp}{Hypothesis}
\newtheorem*{thrm}{Theorem}
\newtheorem*{quest}{Question}
\theoremstyle{remark}

\newcommand{\Zp}{{\mathbb{Z}_p}}
\newcommand{\Qp}{{\mathbb{Q}_p}}
\newcommand{\Fp}{{\mathbb{F}_p}}
\newcommand{\A}{\mathcal{A}}
\newcommand{\B}{\mathcal{B}}
\newcommand{\C}{\mathcal{C}}
\newcommand{\D}{\mathcal{D}}
\newcommand{\E}{\mathcal{E}}
\newcommand{\F}{\mathcal{F}}
\newcommand{\G}{\mathcal{G}}
\newcommand{\I}{\mathcal{I}}
\newcommand{\J}{\mathcal{J}}
\renewcommand{\L}{\mathcal{L}}
\newcommand{\sL}{\mathscr{L}}
\newcommand{\M}{\mathcal{M}}
\newcommand{\sM}{\mathscr{M}}
\newcommand{\N}{\mathcal{N}}
\newcommand{\sN}{\mathscr{N}}
\renewcommand{\O}{\mathcal{O}}
\newcommand{\cP}{\mathcal{P}}
\newcommand{\Q}{\mathcal{Q}}
\newcommand{\R}{\mathcal{R}}
\newcommand{\cS}{\mathcal{S}}
\newcommand{\T}{\mathcal{T}}
\newcommand{\U}{\mathcal{U}}
\newcommand{\sU}{\mathscr{U}}
\newcommand{\sV}{\mathscr{V}}
\newcommand{\V}{\mathcal{V}}
\newcommand{\W}{\mathcal{W}}
\newcommand{\X}{\mathcal{X}}
\newcommand{\Y}{\mathcal{Y}}
\newcommand{\Z}{\mathcal{Z}}
\newcommand{\PFS}{\mathbf{PreFS}}
\newcommand{\h}[1]{\widehat{#1}}
\newcommand{\hA}{\h{A}}
\newcommand{\hK}[1]{\h{#1_K}}
\newcommand{\hsULK}{\hK{\sU(\L)}}
\newcommand{\hsUFK}{\hK{\sU(\F)}}
\newcommand{\hsULnK}{\hK{\sU(\pi^n\L)}}
\newcommand{\hn}[1]{\h{#1_n}}
\newcommand{\hnK}[1]{\h{#1_{n,K}}}
\newcommand{\w}[1]{\wideparen{#1}}
\newcommand{\wK}[1]{\wideparen{#1_K}}
\newcommand{\invlim}{\varprojlim}
\newcommand{\dirlim}{\varinjlim}
\newcommand{\fr}[1]{\mathfrak{{#1}}}
\newcommand{\LRU}[1]{\sU(\mathscr{#1})}
\newcommand{\et}{\acute et}

\newcommand{\ts}[1]{\texorpdfstring{$#1$}{}}
\newcommand{\st}{\mid}
\newcommand{\be}{\begin{enumerate}[{(}a{)}]}
\newcommand{\ee}{\end{enumerate}}
\newcommand{\qmb}[1]{\quad\mbox{#1}\quad}
\let\le=\leqslant  \let\leq=\leqslant
\let\ge=\geqslant  \let\geq=\geqslant

\title{$\w\D$-modules on rigid analytic spaces I}
\author{Konstantin Ardakov}
\address{Mathematical Institute\\University of Oxford\\Oxford OX2 6GG}
\author{Simon Wadsley}
\address{Homerton College, Cambridge, CB2 8PQ}
\begin{abstract}
We introduce a sheaf of infinite order differential operators $\w\D$ on smooth rigid analytic spaces that is a rigid analytic quantisation of the cotangent bundle. We show that the sections of this sheaf over sufficiently small affinoid varieties are Fr\'echet-Stein algebras, and use this to define co-admissible sheaves of $\w\D$-modules. We prove analogues of Cartan's Theorems A and B for co-admissible $\w\D$-modules.
\end{abstract}
\maketitle
\setcounter{tocdepth}{1}  
\tableofcontents
\section{Introduction}
\subsection{Background and motivation} The theory of $\D$-modules goes back over forty years to the work of Sato and Kashiwara for $\D$-modules on manifolds \cite{KashiwaraTh} and to the work of Bernstein for $\D$-modules on algebraic varieties \cite{Bernstein}. Originally introduced as a framework for the algebraic study of partial differential equations there have been also been fundamental applications in the studies of harmonic analysis, algebraic geometry, Lie groups and representation theory. In this paper we attempt to initiate a new theory of $\D$-modules for rigid analytic spaces in the sense of Tate \cite{Tate71}.

In their seminal paper \cite{BB}, Beilinson and Bernstein explained how to study representations of a complex semi-simple Lie algebra $\mathfrak{g}$ via  twisted $\D$-modules on the flag variety $\B$ of the corresponding algebraic group. In particular they established an equivalence between the category of finitely generated modules over the enveloping algebra $U(\mathfrak{g})$ with a fixed regular infinitesimal central character $\chi$ and the category of coherent modules for the sheaf of $\chi$-twisted differential operators on $\B$. 

Our primary motivation for this work is to establish a rigid analytic version of the Beilinson-Bernstein equivalence in order to understand the representation theory of the Arens--Michael envelope $\w{U(\mathfrak{g})}$ of the universal enveloping algebra of a semi-simple Lie algebra $\mathfrak{g}$ over a complete discretely valued field $K$ of mixed characteristic. The Arens--Michael envelope is the completion of $U(\fr{g})$ with respect to all submultiplicative seminorms on $U(\fr{g})$; when $\fr{g}$ is the Lie algebra of a $p$-adic Lie group, $\w{U(\fr{g})}$ occurs as the algebra of locally analytic $K$-valued distributions on this group supported at the identity, and is therefore of interest in the theory of locally analytic representations of $p$-adic groups, developed by Schneider and Teitelbaum \cite{ST4}. We delay the proof of our version of the Beilinson-Bernstein equivalence to a later paper, but see Theorem \ref{BB} below for a precise statement. Here we construct the sheaf $\w{\D}$ on a general smooth rigid analytic space over $K$, and establish some of its basic properties.

\subsection{Rigid analytic quantisation} 
In our earlier work \cite{AW13} we proved an analogous theorem for certain Banach completions of $\w{U(\mathfrak{g})}$ localising onto a smooth formal model  $\h{\B}$ of the flag variety. In this new programme we extend that work in two directions. In the base direction, by working on the rigid analytic flag variety $\B^{an}$ which has a finer topology than a fixed formal model $\h{\B}$, the localisation is more refined and the geometry is more flexible\footnote{In fact with a little extra effort our construction can be localised to the rigid \'etale site but we do not provide the details of that here.}. In the cotangent direction, we no longer fix a level $n$ as we did in \cite{AW13}, and instead work simultaneously with all $n$. This involves using Schneider and Teitelbaum's notions of Fr\'echet--Stein algebras and co-admissible modules introduced in \cite{ST}. 

The definition of a Fr\'echet--Stein algebra is modelled around key properties of Stein algebras; these latter arise as rings of functions on Stein spaces in (complex) analytic geometry. There is a well-behaved abelian category of co-admissible modules defined for each Fr\'echet--Stein algebra; in the case when the algebra in question is ring of global rigid analytic functions on a quasi-Stein rigid analytic space, this category is naturally equivalent to the category of coherent sheaves on this space. It is known \cite{Schmidt09} that $\w{U(\mathfrak{g})}$ is a Fr\'echet-Stein algebra. We view $\w{U(\fr{g})}$ as a quantisation of the algebra of rigid analytic functions on $\mathfrak{g}^\ast$ in much the same way that $U(\mathfrak{g})$ can be viewed as a quantisation of the algebra of polynomial functions on $\mathfrak{g}^\ast$. This is the starting point for our work: our Beilinson--Bernstein style equivalence should have the co-admissible modules for central reductions of $\w{U(\mathfrak{g})}$ on one side.

\subsection{Lie algebroids and completed enveloping algebras}
When working with smooth algebraic varieties in characteristic zero, one can view classical sheaves of differential operators as special cases of sheaves of enveloping algebras of Lie algebroids as in \cite{BB2}, for example. We adopt this more general framework here partly for convenience at certain points of our presentation and partly for the sake of flexibility in future work; in particular we will use it to define sheaves of twisted differential operators in \cite{DhatThree}. In section \ref{Sheaves} below for each Lie algebroid $\sL$ on a rigid analytic space $X$ we construct a sheaf $\w{\sU(\sL)}$ of completed universal enveloping algebras on $X$. When $X$ is smooth we then define $\w{\D}:= \w{\sU(\T)}$. These sheaves $\w{\sU(\sL)}$ may be viewed as quantisations of the total space of the vector bundle $\sL^\ast$. In particular, in this picture, $\w{\D}$ is a quantisation of $T^\ast X$. 

One difficulty with extending the classical work on $\D$-modules to the rigid analytic setting is that there is no known good notion of quasi-coherent sheaf for rigid analytic spaces (but see \cite{BBK} for some recent work in this direction). We resolve that problem here by avoiding it; that is by restricting ourselves to the study of `coherent' modules for our sheaves of rings. Because our sheaves of rings $\w{\sU(\sL)}$ are not themselves coherent the usual notion of coherent sheaves of modules is too strong. However, the sections of our structure sheaves $\w{\sU(\sL)}$ over sufficiently small affinoid subdomains turn out to be Fr\'echet--Stein, so Schneider and Teitelbaum's work shows us how to proceed: we replace the notion of `locally finitely generated' by `locally co-admissible'. Looked at through a particular optic these `locally co-admissible' sheaves do deserve to be seen as if they were coherent.  However, it seems to be necessary to fully develop a theory of micro-local sheaves in our context to make this interpretation precise. 

\subsection{Main results}\label{MainResults}

Our first main result is a non-commutative version of Tate's acyclicity Theorem \cite[Theorem 8.2]{Tate71}.
\begin{MainThm}\label{TateAcy} Suppose that $X$ is a smooth $K$-affinoid variety such that $\T(X)$ is a free $\O(X)$-module. Then \[ \w{\D}(Y):= \w{U(\T(Y))} \] defines a sheaf $\w\D$ of Fr\'echet-Stein algebras on affinoid subdomains of $X$ with vanishing higher \v{C}ech cohomology groups.
\end{MainThm} Here $\w{U(\T(Y))}$ can be concisely defined as the completion of the enveloping algebra $U(\T(Y))$ with respect to all submultiplicative seminorms that extend the supremum seminorm on $\O(Y)$; see Section \ref{DefnOfwUL} below for a more algebraic definition. Our next result involves an appropriate version of completed tensor product $\w\otimes$ developed in Section \ref{Coad}.

\begin{MainThm} \label{MainLoc} Suppose that $X$ is a smooth $K$-affinoid variety such that $\T(X)$ is a free $\O(X)$-module. Then \[ \Loc(M)(U):= \w{\D}(U)\underset{\w{\D}(X)}{\w\otimes}M \] defines a full \emph{exact} embedding $M \mapsto \Loc(M)$ of the category of co-admissible $\w{\D}(X)$-modules into the category of sheaves of $\w{\D}$-modules on affinoid subdomains of $X$, with vanishing higher \v{C}ech cohomology groups.
\end{MainThm}
We can extend $\w\D$ to a sheaf defined on general smooth rigid analytic varieties. Then we prove the following analogue of Kiehl's Theorem \cite{Kiehl} for coherent sheaves of $\O$-modules on rigid analytic spaces.

\begin{MainThm}  \label{MainKiehl} Suppose that $X$ is a smooth analytic variety over $K$. Let $\M$ be a sheaf of $\w{\D}$-modules on $X$. Then the following are equivalent.
\be
\item  There is an admissible affinoid covering $\{X_i\}_{i\in I}$ of $X$ such that $\T(X_i)$ is a free $\O(X_i)$-module, $\M(X_i)$ is a co-admissible $\w{\D}(X_i)$-module and the restriction of $\M$ to the affinoid subdomains of $X_i$ is isomorphic to $\Loc(\M(X_i))$ for each $i\in I$. 
\item  For every affinoid subdomain $U$ of $X$ such that $\T(U)$ is a free $\O(U)$-module, $\M(U)$ is a co-admissible $\w{\D}(U)$-module and $\M(V) \cong \w\D(V) \underset{\w\D(U)}{\w\otimes} \M(U)$ for every affinoid subdomain $V$ of $U$.
\ee
\end{MainThm}

We call a sheaf of $\w{\D}$-modules that satisfies the equivalent conditions of Theorem \ref{MainKiehl} \emph{co-admissible}. Theorems \ref{MainLoc} and \ref{MainKiehl} immediately give the following

\begin{cor} Suppose $X$ is a smooth $K$-affinoid variety such that $\T(X)$ is a free $\O(X)$-module. Then $\Loc$ is an equivalence of abelian categories
\[
\left\{ 
				\begin{array}{c} 
					co\hspace{-0.1cm}-\hspace{-0.1cm}admissible \\ 
					\w\D(X)-\hspace{-0.1cm}modules
				\end{array}
\right\} \cong \left\{
				\begin{array}{c}
				 co\hspace{-0.1cm}-\hspace{-0.1cm}admissible \hspace{0.1cm} sheaves \hspace{0.1cm} of\\ 
				  \w{\D} \hspace{-0.1cm}-\hspace{-0.1cm}modules \hspace{0.1cm} on \hspace{0.1cm} X
				\end{array}
\right\}.
\]
\end{cor}

In fact we prove each of these statements in greater generality with $\T$ replaced by any Lie algebroid on any reduced rigid analytic space over $K$, and for right modules as well as left modules. 

\subsection{Future and related work}\label{Future}
We plan to explain in the future how parts of the vast classical theory of $\D$-modules generalise to our setting with the results contained in this work being merely the leading edge of what is to come. In particular in \cite{DhatTwo} we will prove the following analogue of Kashiwara's equivalence.
\begin{MainThm}\label{Kashiwara} Let $Y$ be a smooth closed analytic subvariety of a smooth rigid analytic variety $X$. There is a natural equivalence of categories
\[
\left\{ 
				\begin{array}{c} 
					co\hspace{-0.1cm}-\hspace{-0.1cm}admissible\hspace{0.1cm} sheaves \hspace{0.1cm} of\\
					\hspace{0.1cm}\w\D-\hspace{-0.1cm}modules\hspace{0.1cm}on\hspace{0.1cm} Y
				\end{array}
\right\} \cong \left\{
				\begin{array}{c}
				 co\hspace{-0.1cm}-\hspace{-0.1cm}admissible \hspace{0.1cm} sheaves \hspace{0.1cm} of \\ 
				 \hspace{0.1cm} \w{\D} \hspace{-0.1cm}-\hspace{-0.1cm}modules\hspace{0.1cm} on \hspace{0.1cm} X \hspace{0.1cm} supported \hspace{0.1cm} on \hspace{0.1cm} Y
				\end{array}
\right\}.
\]
\end{MainThm} 

In future work \cite{DhatThree}, we will prove an analogue of Beilinson and Bernstein's localisation theorem of \cite{BB} for twisted $\w{\D}$-modules on $\B^{\an}$. For the sake of brevity, we will only state the version of this result for un-twisted $\w{\D}$-modules here.
\begin{MainThm}\label{BB} Let $\mathbf{G}$ be a connected split reductive group over $K$ with Lie algebra $\mathfrak{g}$, let $\B^{\an}$ be the rigid analytic flag variety and let $Z(\mathfrak{g})$ be the centre of $U(\fr{g})$. Then there is an equivalence of abelian categories 
\[
\left\{ 
				\begin{array}{c} 
					co\hspace{-0.1cm}-\hspace{-0.1cm}admissible \\ 
					\w{U(\fr{g})} \otimes_{Z(\fr{g})} K \hspace{-0.1cm}-\hspace{-0.1cm}modules
				\end{array}
\right\} \cong \left\{
				\begin{array}{c}
				 co\hspace{-0.1cm}-\hspace{-0.1cm}admissible \hspace{0.1cm} sheaves \hspace{0.1cm}of\\ 
				  \w\D\hspace{-0.1cm}-\hspace{-0.1cm}modules \hspace{0.1cm} on \hspace{0.1cm} \B^{\an} 
				\end{array}
\right\}.
\]
 \end{MainThm}

We hope, perhaps even expect, that this work will have wider applications. Certainly it seems likely that the study of $p$-adic differential equations will be synergetic with our work. Also, much as the theory of algebraic $\D$-modules was influential for the field of non-commutative algebraic geometry, this work might point towards a non-commutative rigid analytic geometry (see also \cite{Soi1}).

It is appropriate to mention here the body of work by Berthelot and others begun in \cite{Berth} that considers sheaves of arithmetic differential operators on smooth formal schemes $\mathcal{X}$ over $W(k)$. There are points of connection between our work and Berthelot's but the differences are substantial.  We also note that Patel, Schmidt and Strauch have begun a programme \cite{PSS} of localising locally analytic representations of non-compact semi-simple $p$-adic Lie groups onto Bruhat-Tits buildings. Whilst their motivation is similar to ours there are again significant differences between our approach and theirs. 

\subsection{Intermediate constructions}

In order to construct our sheaves $\w{\sU(\sL)}$ we first define some intermediate objects that may well prove to be of interest in their own right. 

Let $\R$ denote the ring of integers of our ground field $K$, and fix a non-zero non-unit $\pi \in \R$. Let $X$ be a reduced \emph{affinoid} variety over $K$. Given an affine formal model $\A$ in $\O(X)$ and an $(\R,\A)$-Lie algebra $\L$ we define a $G$-topology $X_w(\L)$ on $X$ consisting of those affinoid subdomains $Y$ of $X$ such that $\O(Y)$ has an affine formal model $\B$ with the property that the unique extension of the natural action of $\L$ on $\O(X)$ to an action on $\O(Y)$ preserves $\B$. We call these affinoid subdomains \emph{$\L$-admissible}.

For example if $X=\Sp K\langle x\rangle$, $\A=\R\langle x\rangle$, and $\L=\A\partial_x$ then the closed disc $Y \subset X$ of radius $|p|^{1/p}$ centred at zero is $\L$-admissible because $\R\langle x,x^p/p\rangle$ is an $\L$-stable affine formal model in $\O(Y)$.  The smaller closed disc of radius $|p|$ is not $\L$-admissible, however it \emph{is} $p\L$-admissible.

A key result due to Rinehart \cite[Theorem 3.1]{Rinehart} that underlies much of our work can be viewed as a generalisation of the Poincar\'e-Birkhoff-Witt theorem to the setting of $(\R,\A)$-algebras. To apply this theorem directly to an enveloping algebra $U(\L)$, the $(\R,\A)$-algebra $\L$ is required to be \emph{smooth}. When this is the case, we construct a sheaf of Noetherian Banach algebras $\hK{\sU(\L)}$ on $X_w(\L)$. 

We would have liked to prove that the restriction maps $\hK{\sU(\L)}(Y)\to \hK{\sU(\L)}(Z)$ are flat whenever $Z \subset Y$ are $\L$-admissible affinoid subdomains of $X$. Because we were unable to do this, we instead define a weaker \emph{$\L$-accessible} $G$-topology $X_{ac}(\L)$ on $X$, and prove that if $Z\subset Y$ are $\L$-accessible then $\hK{\sU(\L)}(Z)$ is a flat $\hK{\sU(\L)}(Y)$-module on both sides. Since every affinoid subdomain of $X$ is $\pi^n\L$-accessible for sufficiently large $n$, this turns out to be sufficient for our purposes.

Now, the $X_w(\pi^n \L)$ form an increasing chain of $G$-topologies on $X$ and every affinoid subdomain $Y$ of $X$ lives is $X_w(\pi^n\L)$ for sufficiently large $n$, so the formula $\invlim \hK{\sU(\pi^n\L)}(Y)$ defines a presheaf of $K$-algebras $\w{\sU(K\otimes_\R \L)}$ on the affinoid subdomains of $X$. We show that this presheaf is actually a sheaf, and that its sections over affinoid subdomains $Y$ are Fr\'echet--Stein in the sense of \cite{ST} with respect to the family $(\hK{\sU(\pi^n\L)}(Y))_{n\gg 0}$. We then use a version of the Comparison Lemma to extend this construction to a sheaf on every reduced rigid analytic space $X$.

\subsection{Structure of the paper}

The main body of the paper begins in Section \ref{TateAcyc} where we define and study the $G$-topology $X_w(\L)$ associated to a $K$-affinoid variety $X$ with an affine formal model $\A$ and a smooth $(\R,\A)$-Lie algebra $\L$ as explained above. The main result of that section is that the presheaf $\hK{\sU(\L)}$ on $X_w(\L)$ defined therein is a sheaf with no higher cohomology. In Section \ref{Flatness} we prove that various continuous $K$-algebra homomorphisms that arise as restriction maps in the sheaves $\hsULK$ on $X_{ac}(\L)$ are flat. In Section \ref{CohMod} we prepare the way for Theorems \ref{MainLoc} and \ref{MainKiehl} by proving preliminary versions for the sheaves $\hsULK$ on $X_{ac}(\L)$. 

In Section \ref{FrechEnv} we begin our study of Fr\'echet--Stein algebras. In particular we give a functorial construction that associates a Fr\'echet--Stein algebra $\w{U(L)}$ to each coherent $(K,A)$-Lie algebra $L$ with $A$ affinoid. We do this via a more general construction that associates a Fr\'echet--Stein algebra to every deformable $\R$-algebra with commutative Noetherian associated graded ring. Then in Section \ref{Coad} we define a base change functor between categories of co-admissible modules over Fr\'echet--Stein algebras $U$ and $V$ that possess a suitable $U-V$-bimodule. 

In Sections \ref{CoadsUL} and \ref{Sheaves} we put all this together in order to prove Theorems \ref{TateAcy}--\ref{MainKiehl}.  More precisely, Theorems \ref{TateAcy} and \ref{MainLoc} are special cases of Theorems \ref{Envaff} and \ref{Loc}, whereas Theorem \ref{MainKiehl} and its Corollary are special cases of Theorem \ref{CoadmSheaves} and Theorem \ref{TwoCors}, respectively.

\subsection{Acknowledgements} The authors are very grateful to Ian Grojnowski for introducing them to rigid analytic geometry and to localisation methods in representation theory. They would also like to thank Ahmed Abbes, Oren Ben-Bassat, Joseph Bernstein, Kenny Brown, Simon Goodwin, Ian Grojnowski, Michael Harris, Florian Herzig, Christian Johannson,  Minhyong Kim, Kobi Kremnitzer, Shahn Majid, Vytas Pa\v{s}k\={u}nas, Tobias Schmidt, Peter Schneider, Wolfgang Soergel, Matthias Strauch, Catharina Stroppel, Go Yamashita and James Zhang for their interest in this work. The first author was supported by EPSRC grant EP/L005190/1.

\subsection{Conventions} Throughout the remainder of this paper $K$ will denote a complete discrete valuation field with valuation ring $\R$ and residue field $k$. We fix a non-zero non-unit element $\pi$ in $\R$. If $\M$ is an $\R$-module, then $\h{\M}$ denotes the $\pi$-adic completion of $\M$. The term "module" will mean \emph{left} module, unless explicitly stated otherwise.

\section{Generalities}
\subsection{Enveloping algebras of Lie--Rinehart algebras}
Let $R$ be a commutative base ring, and let $A$ be a commutative $R$-algebra. A \emph{Lie--Rinehart algebra}, or more precisely, an ($R,A$)-\emph{Lie algebra} is a pair $(L,\rho)$ where
\begin{itemize}	
\item $L$ is an $R$-Lie algebra and an $A$-module, and
\item $\rho \colon L \to \Der_R(A)$ is an $A$-linear Lie algebra homomorphism
\end{itemize}
called the \emph{anchor map}, such that $[x,ay]=a[x,y]+\rho(x)(a)y$ for all $x,y \in L$ and $a \in A$; see \cite{Rinehart}. We will frequently abuse notation and simply denote $(L,\rho)$ by $L$ whenever the anchor map $\rho$ is understood.

For every $(R,A)$-Lie algebra $L$ there is an associative $R$-algebra $U(L)$ called the \emph{enveloping algebra} of $L$, which comes equipped with canonical homomorphisms 
\[i_A \colon A \to U(L)\qmb{and}i_L : L \to U(L)\]
of $R$-algebras and $R$-Lie algebras respectively, satisfying
\[i_L(ax) = i_A(a)i_L(x) \qmb{and} [i_L(x), i_A(a)] = i_A(\rho(x)(a)) \qmb{for all} a \in A, x \in L.\]
The enveloping algebra $U(L)$ enjoys the following universal property: whenever $j_A \colon A \to S$ is an $R$-algebra homomorphism and $j_L : L \to S$ is an $R$-Lie algebra homomorphism such that
\[j_L(ax) = j_A(a)j_L(x) \qmb{and} [j_L(x), j_A(a)] = j_A(\rho(x)(a)) \qmb{for all} a \in A, x \in L,\]
there exists a unique $R$-algebra homomorphism $\varphi \colon U(L) \to S$ such that
\[ \varphi \circ i_A = j_A \qmb{and} \varphi \circ i_L = j_L.\]
It is easy to show \cite[\S 2]{Rinehart} that $i_A \colon A \to U(L)$ is always injective, and we will always identify $A$ with its image in $U(L)$ via $i_A$.

If $(L,\rho), (L',\rho')$ are two $(R,A)$-Lie algebras then a \emph{morphism of $(R,A)$-Lie algebras} is an $A$-linear map $f:L \to L'$ that is also a morphism of $R$-Lie algebras and satisfies $\rho'f=\rho$. 

A morphism of $(R,A)$-Lie algebras $f\colon L\to L'$ induces an $R$-algebra homomorphism $U(f)\colon U(L)\to U(L')$ via $U(f)(a)=a$ for $a\in A$ and $U(f)(i_{L}x)=i_{L'}(f(x))$ for $x\in L$. So in this way, $U$ defines a functor from $(R,A)$-Lie algebras to associative $R$-algebras.

\begin{defn} We say that an $(R,A)$-Lie algebra $L$ is \emph{coherent} if it is coherent as an $A$-module. We say that $L$ is \emph{smooth} if in addition it is projective as a $A$-module.
\end{defn}
\subsection{Base extensions of Lie--Rinehart algebras}\label{BaseExt} Let $A$ and $B$ be commutative $R$-algebras and let $\varphi \colon A \to B$ be an $R$-algebra homomorphism. If $L$ is an $(R,A)$-Lie algebra, the $B$-module $B \otimes_AL$ will not be an $(R,B)$-Lie algebra, in general.  However this is true in many interesting situations.

\begin{lem} Suppose that the anchor map $\rho \colon L \to \Der_R(A)$ lifts to an $A$-linear Lie algebra homomorphism $\sigma \colon L \to \Der_R(B)$ in the sense that
\[ \sigma(x) \circ \varphi = \varphi \circ \rho(x) \qmb{for all} x \in L.\]
Then $(B \otimes_AL, 1\otimes \sigma)$ with the natural $B$-linear structure is an $(R,B)$-Lie algebra in a unique way.
\end{lem}

\begin{proof} Write $x \cdot b := \sigma(x)(b)$ and $bx := b\otimes x$ for all $x \in L$ and $b \in B$. Following \cite[(3.5)]{Rinehart}, we define a bracket operation on $B\otimes_A L$ in the only possible way as follows:
\[ [ b  x , b'  x' ] := bb'  [x,x'] - b' (x'\cdot b)  x + b (x \cdot b')  x'\]
for all $b,b' \in B$ and $x,x' \in L$. It is straightforward to verify that this bracket is well-defined, skew-symmetric, and satisfies
\[ [b  x, c(b'  x')] = c[b  x, b'  x'] + (1 \otimes \sigma)(b  x)(c)b'x'\]
for all $b,b',c \in B$ and $x,x' \in L$. Note that if $x,y,z\in L$ and $b \in B$ then
\[ [[1  x, 1  y], b  z] + [[1  y, b  z], 1  x] + [[b  z, 1  x], 1  y] = ([x,y] \cdot b - x \cdot (y \cdot b) + y \cdot (x \cdot b))  z\]
so the condition that $\sigma : L \to \Der_k(B)$ is a Lie homomorphism is necessary for the Jacobi identity to hold. A longer, but still straightforward, computation shows that this condition is also sufficient. 
\end{proof}

\begin{cor} Suppose that $\psi\colon \Der_R(A)\to \Der_R(B)$ is an $A$-linear homomorphism of $R$-Lie algebras such that $\psi(u)\circ \varphi= \varphi\circ u$ for each $u\in \Der_R(A)$. There is a natural functor $B\otimes_A -$ from $(R,A)$-Lie algebras to $(R,B)$-Lie algebras sending $(L,\rho)$ to $(B\otimes_A L, 1\otimes \psi\rho)$.
\end{cor}
\begin{proof} Suppose that $(\rho,L)$ and $(\rho',L')$ are $(R,A)$-Lie algebras and $f\colon L\to L'$ is a morphism of $(R,A)$-Lie algebras. Give $(B\otimes_AL, 1\otimes \psi\rho)$ and $(B\otimes_A L', 1\otimes \psi\rho')$ the structures of $(R,B)$-Lie algebras guaranteed by the Lemma; we have to show that $1 \otimes f : B \otimes_A L \to B \otimes_A L'$ is then a morphism of $(R,B)$-Lie algebras. 

It is $B$-linear and satisfies $(1 \otimes \psi\rho') \circ (1 \otimes f) = 1 \otimes \psi \rho$ because $\rho' f = \rho$. Now if $b,c \in B$ and $x,y \in L$ then
\begin{eqnarray*} 
[(1 \otimes f)(bx), (1 \otimes f)(cy)] &=& bc[f(x),f(y)]- c(  y \cdot b)f(x)+ b(  x \cdot c) f(y) \\ 
& = & (1\otimes f)([bx,cy]). 
\end{eqnarray*} Thus $1 \otimes f$ is an $R$-Lie algebra homomorphism.
\end{proof}

\subsection{Rinehart's Theorem}\label{RineTh}
Let $\Sym(L)$ denote the symmetric algebra of the $A$-module $L$. Rinehart proved \cite[Theorem 3.1]{Rinehart} that there is always a surjection
\[ \Sym(L) \twoheadrightarrow \gr U(L)\]
which is even an isomorphism whenever $L$ is smooth. Therefore $U(L)$ is a (left and right) Noetherian ring whenever $A$ is Noetherian and $L$ is a finitely generated $A$-module; we will use this basic fact without further mention in what follows.

\begin{prop} Let $\varphi\colon A \to B$ be a homomorphism of commutative $R$-algebras and let $(L, \rho)$ be a smooth ($R,A$)-Lie algebra. Suppose that $\rho \colon L \to \Der_R(A)$ lifts to an $A$-linear Lie algebra homomorphism $\sigma \colon L \to \Der_R(B)$. Then there are natural isomorphisms
\[ B \otimes_A U(L) \to U(B \otimes_A L) \qmb{and} U(L) \otimes_AB \to U(B \otimes_A L)\]
of filtered left $B$-modules and filtered right $B$-modules, respectively.
\end{prop}
\begin{proof} The pair $(B \otimes_AL, 1 \otimes \sigma)$ is an ($R,B$)-Lie algebra by Lemma \ref{BaseExt}. The universal property of $U(L)$ induces a homomorphism of filtered $R$-algebras \[U(\varphi) \colon U(L) \to U(B \otimes_AL)\]
such that $U(\varphi)(i_L(x)) = i_{B \otimes_AL}(1 \otimes x)$ for all $x\in L$. Since $U(\varphi)$ is left $A$-linear, we obtain a filtered left $B$-linear homomorphism
\[ 1 \otimes U(\varphi) \colon B \otimes_A U(L) \longrightarrow U( B \otimes_A L).\]
By \cite[Theorem 3.1]{Rinehart}, its associated graded can be identified with the natural map
\[ B \otimes_A \Sym(L) \longrightarrow \Sym(B \otimes_A L)\]
which is an isomorphism by \cite[Chapter III, \S 6, Proposition 4.7]{BourbakiAlgI}. The isomorphism $U(\varphi) \otimes 1 \colon U(L) \otimes_AB \to U(B \otimes_AL)$ of right $B$-modules is established similarly.
\end{proof}

\subsection{Lifting derivations of affinoid algebras}\label{LiftEtale}
Recall, \cite[\S3.3]{Berkovich93}, that if $A\to B$ is a morphism of affinoid algebras then there is a finitely generated $B$-module 
\[\Omega_{B/A}\] 
such that for any Banach $B$-module $M$ there is a natural isomorphism \[ \Hom_B(\Omega_{B/A},M)\cong \Der^b_A(B,M)\] where $\Der^b_A(B,M)$ denotes the set of $A$-linear \emph{bounded} derivations from $B$ to $M$. Note that every $K$-linear derivation from $B$ to a \emph{finitely generated} $B$-module $M$ is automatically bounded; this follows from the proof of \cite[Theorem 3.6.1]{FvdPut}. So in particular, $\Der^b_K(B,B) = \Der_K(B)$.

\begin{lem} Let $\varphi\colon A \to B$ be an \'etale morphism of $K$-affinoid algebras. Then there is a unique $A$-linear map
\[\psi : \Der_K(A)\to \Der_K(B)\] such that $\psi(u)\circ \varphi=\varphi\circ u$ for each $u\in \Der_K(A)$. Moreover $\psi$ is a homomorphism of $K$-Lie algebras.
\end{lem}
\begin{proof} By \cite[Corollary 2.1.8/3 and Theorem 6.1.3/1]{BGR},  $\varphi : A \to B$ is bounded. Hence, composition with $\varphi$ induces $A$-linear maps 
\[\Der_K(A) \stackrel{\alpha}{\longrightarrow} \Der_K^b(A,B) \stackrel{\beta}{\longleftarrow} \Der_K(B).\] 
Since $A\to B$ is \'etale, \cite[Proposition 3.5.3(i)]{Berkovich93} guarantees that the natural morphism $B\otimes_A \Omega_{A/K}\to \Omega_{B/K}$ is an isomorphism. Taking $B$-linear duals shows that the restriction map 
\[ \beta : \Der_K(B) \to \Der_K^b(A,B)\]
is also an isomorphism and therefore every $K$-linear derivation of $B$ is determined by its restriction to $A$. We therefore obtain a unique $A$-linear map
\[ \psi := \beta^{-1} \circ \alpha : \Der_K(A) \to \Der_K(B)\]
such that $\psi(u) \circ \varphi = \varphi \circ u$ for all $u \in \Der_K(A)$. If $u,v \in \Der_K(A)$ then the $K$-linear derivations $\psi([u,v])$ and $[\psi(u), \psi(v)]$ of $B$ agree on the image of $A$ in $B$ and therefore are equal. Hence $\psi$ is a Lie homomorphism.
\end{proof}

Combining the Lemma with Corollary \ref{BaseExt} gives the following

\begin{cor} Let $A \to B$ be an \'etale morphism of $K$-affinoid algebras and let $L$ be a $(K, A)$-Lie algebra. Then there is a unique structure of a $(K,B)$-Lie algebra on $B \otimes_AL$ with its natural $B$-module structure such that the natural map $L \to B \otimes_A L$ is a $K$-Lie algebra homomorphism and the diagram \[ \xymatrix{ L \ar[r]^{\rho_L} \ar[d] & \Der_K(A) \ar[d]^\psi\\ B\otimes_A L \ar[r]_{\rho_{B\otimes_A L}} & \Der_K(B)} \]  commutes. Moreover this defines a canonical functor from $(K,A)$-Lie algebras to $(K,B)$-Lie algebras.
\end{cor}

\subsection{Lemma}\label{BddTorsionHat} Let $C^\bullet$ be a complex of flat $\R$-modules with bounded torsion cohomology. Then $H^q(\widehat{C^\bullet}) \cong H^q(C^\bullet)$ for all $q$, and $\widehat{C^\bullet} \otimes_\R K$ is exact.
\begin{proof} Since $C^\bullet$ has no $\pi$-torsion by assumption, for each $n, m \geq 0$ we have a commutative diagram of complexes of $\R$-modules with exact rows:
\[\xymatrix{ 0 \ar[r] & C^\bullet \ar[r]^{\pi^{n+m}} \ar[d]^{\pi^m} & C^\bullet \ar[r] \ar@{=}[d] & C^\bullet / \pi^{n+m} C^\bullet \ar[r]\ar[d] & 0 \\
0 \ar[r] & C^\bullet \ar[r]_{\pi^n} & C^\bullet \ar[r] & C^\bullet / \pi^n C^\bullet \ar[r] & 0. }\]
Now fix $q$, choose $N \geq 0$ such that $H^q(C^\bullet)$ and $H^{q+1}(C^\bullet)$ are killed by $\pi^N$ and let $n,m \geq N$. Applying the long exact sequence of cohomology produces another commutative diagram with exact rows:
\[\xymatrix{ 0 \ar[r] & H^q(C^\bullet) \ar[r] \ar@{=}[d]& H^q(C^\bullet / \pi^{n+m} C^\bullet) \ar[r] \ar[d] & H^{q+1}(C^\bullet)\ar[r]\ar[d]^{\pi^m} & 0 \\
0 \ar[r] & H^q(C^\bullet) \ar[r]  & H^q(C^\bullet / \pi^n C^\bullet) \ar[r] & H^{q+1}(C^\bullet)\ar[r] & 0. }\]
Consider this diagram as a short exact sequence of towers of $\R$-modules. Since the vertical arrow on the right is zero for $m \geq N$ by assumption, and the vertical arrow on the left is an isomorphism, the long exact sequence associated to the inverse limit functor $\invlim$ shows that
\[ H^q(C^\bullet) \cong \invlim H^q(C^\bullet / \pi^n C^\bullet) \qmb{and} {\invlim}^1 H^q(C^\bullet / \pi^n C^\bullet) = 0 \qmb{for all} q.\]
Because the maps in the tower of complexes $(C^\bullet / \pi^n C^\bullet)_n$ are surjective, this tower satisfies the Mittag-Leffler condition. The cohomological variant of \cite[Theorem 3.5.8]{Wei1995} implies that
\[\invlim H^q(C^\bullet / \pi^n C^\bullet) \cong H^q(\widehat{C^\bullet}) \qmb{for all} q.\]
Therefore $\widehat{C^\bullet}$ has $\pi$-torsion cohomology.
\end{proof}

\subsection{Torsion in \ts{U(\L)}}\label{TorsUL}
Let $\A$ be a commutative Noetherian $\R$-algebra, and let $\L$ be a coherent $(\R,\A)$-Lie algebra. Let $\overline{\L}$ denote the image of $\L$ in $\L \otimes_\R K$; this is again an $(\R,\A)$-Lie algebra which is now flat as an $\R$-module. 

Let $\overline{U(\L)}$ denote the image of $U(\L)$ in $U(\L)\otimes_\R K$; note that unless $\L$ happens to be smooth, the $\pi$-torsion submodule of $U(\L)$ may well be non-zero. In any case, it is easy to see that there is a commutative diagram of $\R$-algebra homomorphisms with surjective arrows
\[\xymatrix{ U(\L) \ar@{>>}[r]\ar@{>>}[d] & U(\overline{\L}) \ar@{>>}[d]\\ \overline{U(\L)} \ar@{.>>}[r] &\overline{U(\overline{\L})}. }\]
Note that because $U(\L \otimes_\R K) \cong U(\L) \otimes_\R K$, the bottom arrow in this diagram is actually an isomorphism.
\begin{lem} The functor $X \mapsto \hK{X} = \h{X}\otimes_\R K$ transforms each arrow in the above diagram into an isomorphism.
\end{lem}
\begin{proof}  The kernel of $U(\L) \to \overline{U(\L)}$ is a finitely generated left ideal $T$ in $U(\L)$ since $U(\L)$ is Noetherian. Since $T\otimes_\R K = 0$ by construction, we see that $\pi^n \cdot T = 0$ for some $n\geq 0$. The sequence $0 \to \h{T} \to \h{U(\L)} \to \h{\overline{U(\L)}} \to 0$ is exact by \cite[\S 3.2.3(ii)]{Berth}, and $\pi^n \cdot \h{T} = 0$, so $\hK{U(\L)} \to \hK{\overline{U(\L)}}$ is an isomorphism.  This deals with the vertical arrows, and the result follows.
\end{proof}

We will also require the following elementary result concerning flat modules.
\subsection{Lemma}\label{FlatMethod} Let $S \to T$ be a ring homomorphism. Let $u \in T$ be a left regular element and suppose that
\be
\item $T$ is a flat right $S$-module,
\item $T \otimes_SM$ is $u$-torsion-free for all finitely generated left $S$-modules $M$. \ee
Then $W := T / u T$ is also a flat right $S$-module.
\begin{proof} Let $M$ be a finitely generated $S$-module and pick a projective resolution $P_\bullet \twoheadrightarrow M$ of $M$. Since $T_S$ is flat, $T \otimes_S P_\bullet \twoheadrightarrow T \otimes_S M$ is a projective resolution so
\[ \Tor_1^S(W,M) = H_1(W \otimes_S P_\bullet) = H_1(W \otimes_T (T \otimes_S P_\bullet)) \cong \Tor_1^T(W, T \otimes_S M).\]
The short exact sequence $0 \to T \stackrel{u\cdot}{\longrightarrow} T \to W \to 0$ induces the long exact sequence
\[0 = \Tor_1^T(T, T \otimes_SM) \to \Tor_1^T(W, T \otimes_S M) \to T \otimes_S M \stackrel{u \cdot}{\longrightarrow} T \otimes_S M,\]
so $\Tor_1^S(W,M) = \Tor_1^T(W, T \otimes_S M)$ vanishes by assumption (b). Hence $W$ is a flat right $S$-module by \cite[Proposition 3.2.4]{Wei1995}.
\end{proof}

\section{Tate's Acyclicity Theorem for \ts{\hK{\sU(\L)}}}\label{TateAcyc} 

\subsection{\ts{\L}-stable affine formal models}\label{AdmRalg}
Recall \cite[\S 1]{BL1} that an \emph{admissible $\R$-algebra} is a commutative $\R$-algebra which is topologically of finite type and flat over $\R$. 

If $A$ is a reduced $K$-affinoid algebra and $\A$ is an admissible $\R$-algebra then we say that $\A$ is an \emph{affine formal model in $A$} if $A \cong \A \otimes_\R K$. Note that any affine formal model in this sense is automatically reduced.

\begin{lem} Let $\A$ be an affine formal model in a reduced $K$-affinoid algebra $A$.
\be 
\item $\A$ is contained in $A^\circ$, and $A^\circ$ is a finitely generated $\A$-module.
\item $A^\circ$ is an affine formal model in $A$.
\item Every subring $\A'$ of $A$ such that $\A \subset \A' \subseteq A^\circ$ is also an affine formal model in $A$.
\ee
\end{lem}
\begin{proof} (a) Because $\A$ is topologically finitely generated over $\R$, we can find a Tate algebra $T = K \langle x_1,\cdots, x_n \rangle$ and a surjective homomorphism $\sigma : T^\circ \twoheadrightarrow \A$. Then $\sigma_K : T \twoheadrightarrow A$. Since $K$ is discretely valued and $A$ is reduced, \cite[Theorem 3.5.6]{FvdPut} implies that $A^\circ$ contains $\sigma_K(T^\circ) = \A$, and is finitely generated over as a module over $\sigma_K(T^\circ) = \A$ .

(b) We can find some Tate algebra $T = K \langle x_1,\cdots, x_n \rangle$ and a surjective homomorphism $\varphi : T \twoheadrightarrow A$. Then $\varphi(T^\circ)$ is topologically finitely generated over $\R$, and $A^\circ$ is a finitely generated $\varphi(T^\circ)$-module by \cite[Theorem 3.5.6]{FvdPut}. So $A^\circ$ is also topologically finitely generated over $\R$, and it is flat over $\R$ being contained in the $K$-vector space $A$.

(c) Since $\A$ is Noetherian and $A^\circ$ is finitely generated as an $\A$-module by part (a), $\A'$ is also finitely generated as an $\A$-module. Therefore it is an admissible $\R$-algebra and $\A' \otimes_\R K = A$.
\end{proof}

Let $\sigma : A \to B$ be an \'etale morphism of reduced $K$-affinoid algebras and let $\A$ be an affine formal model in $A$. Let $\L$ be an $(\R,\A)$-Lie algebra. Note that because $\sigma : A \to B$ is \'etale, the action of $L := \L\otimes_\R K$ on $A$ lifts automatically to $B$ by Corollary \ref{LiftEtale}.

\begin{defn} Let $\B$ be an affine formal model in $B$. We say that $\B$ is \emph{$\L$-stable} if $\sigma(\A) \subset \B$ and the action of $\L$ on $\A$ lifts to $\B$. We say that $\sigma : A \to B$ is \emph{$\L$-admissible} if there exists at least one $\L$-stable affine formal model $\B$ in $B$.
\end{defn}

\begin{prop} Let $\sigma : A \to B$ be an $\L$-admissible morphism of reduced $K$-affinoid algebras, let $\A$ be an affine formal model in $A$ and let $\L$ be an $(\R,\A)$-Lie algebra.
\be
\item The set of $\L$-stable affine formal models in $B$ is closed under multiplication.
\item There exists a unique largest $\L$-stable affine formal model $B^\diamond$ in $B$.
\item $\sigma(A^\diamond) \subseteq B^\diamond$.
\ee
\end{prop}
\begin{proof} (a) Any two $\L$-stable affine formal models in $B$ are both contained in $B^\circ$ by part (a) of the Lemma because $B$ is reduced. So their product is again an affine formal model in $\B$ by part (c) of the Lemma. This product is also $\L$-stable because $\L$ acts on $B$ by $\R$-linear derivations. 

(b) This follows from part (a), part (a) of the Lemma and the fact that any affine formal model is a Noetherian ring.

(c) The $\R$-subalgebra $\sigma(A^\diamond) B^\diamond$ contains $B^\diamond$ and is $\L$-stable, therefore it is an $\L$-stable affine formal model by part (c) of the Lemma. Therefore it must be contained in $B^\diamond$ by the maximality of $B^\diamond$.
\end{proof}

\subsection{\ts{\L}-admissible subdomains and the functor \ts{\hsULK}}\label{SheafFQ} Let $X$ be a $K$-affinoid variety. Recall from \cite[\S 9.1.4]{BGR} the \emph{strong $G$-topology} $X_{\rig}$ consisting of the \emph{admissible open subsets} of $X$ and admissible coverings, and the \emph{weak $G$-topology} $X_w$ on $X$ consisting of the \emph{affinoid subdomains} of $X$ and \emph{finite} coverings by affinoid subdomains.

\begin{lem} Let $X$ be a $K$-affinoid variety and let $Q$ be a flat $\O(X)^\circ$-module. For every admissible open subset $Y$ of $X$, define
\[\F_Q(Y) := \widehat{ \O(Y)^\circ \otimes_{\O(X)^\circ} Q} \otimes_\R K.\]
Then $\F_Q$ is a sheaf on $X_{\rig}$.
\end{lem}
\begin{proof} We know that $\O^\circ$ is a sheaf on $X_{\rig}$ by \cite[Example 8.2.2(3)]{FvdPut}. The functor $A \mapsto \widehat{A \otimes_{\O(X)^\circ} Q}\otimes_\R K$ is left exact because $Q$ is a flat $\O(X)^\circ$-module. Therefore $\F_Q$ is also a sheaf on $X_{\rig}$.
\end{proof}

\begin{defn} Let $X$ be a reduced $K$-affinoid variety, and let $\L$ be an $(\R,\A)$-Lie algebra for some affine formal model $\A$ in $\O(X)$. 
\be
\item We say that an affinoid subdomain $Y$ of $X$ is \emph{$\L$-admissible} if the pullback on functions $f^\sharp : \O(X) \to \O(Y)$ is $\L$-admissible. 
\item For any $\L$-admissible affinoid subdomain $Y$ of $X$, we define
\[\hsULK(Y) := \widehat{ U(\O(Y)^\diamond \otimes_{\A} \L) } \otimes_\R K.\]
\ee\end{defn}

Note that $\O(Y)^\diamond \otimes_{\A} \L$ is an $(\R, \O(Y)^\diamond)$-Lie algebra by Lemma \ref{BaseExt}, so $\hsULK(Y)$ is an associative $K$-Banach algebra. 

\subsection{Fibre products}\label{FibreProds}

\begin{lem} Let $B \leftarrow A \stackrel{\sigma}{\to} C \stackrel{\tau}{\to} D$ be a diagram of reduced $K$-affinoid algebras with \'etale morphisms, let $\A$ be an affine formal model in $A$ and let $\L$ be an $(\R,\A)$-Lie algebra. Let $\B$ and $\C$ be $\L$-stable affine formal models in $B$ and $C$, respectively.
\be
\item $\C \otimes_{\A} \L$ is a $(\R, \C)$-Lie algebra.
\item $\tau \sigma$ is $\L$-admissible if and only if $\tau$ is $\C \otimes_{\A} \L$-admissible.
\item Let $\D$ be the largest $\C \otimes_{\A} \L$-stable affine formal model in $D$. Then $\D = D^\diamond$.
\item The image of $\B \widehat{\otimes}_{\A} \C$ in $B \widehat{\otimes}_A C$ is an $\L$-stable affine formal model in $B \widehat{\otimes}_AC$.
\ee
\end{lem}
\begin{proof} (a) This follows from Lemma \ref{BaseExt}.

(b) Suppose $\tau \sigma$ is $\L$-admissible. Choose an $\L$-stable affine formal model $\D'$ in $D$ and let $\D'' = \tau(\C) \D'$. Then $\D''$ is again an $\L$-stable affine formal model in $D$ by Lemma \ref{AdmRalg}(c), and it is a $\C$-submodule of $D$ via $\tau$. Hence $\D''$ is $\C \otimes_{\A} \L$-stable and $D$ is $\C \otimes_{\A} \L$-admissible. The converse is clear.

(c) $\D$ is $\L$-stable so $\D \subseteq D^\diamond$. On the other hand $\tau(\C)D^\diamond$ is a $\C \otimes_{\A} \L$-stable affine formal model in $D$ so $D^\diamond \subseteq \tau(\C)D^\diamond \subseteq \D$.

(d) Let $\overline{\B \widehat{\otimes}_{\A} \C}$ denote the image of $\B \widehat{\otimes}_{\A} \C$ in $B \widehat{\otimes}_A C$. Then the fibre product $\Spf(\B) \times_{\Spf(\A)} \Spf(\C)$ \emph{in the category of admissible formal schemes} is $\Spf(\overline{\B \widehat{\otimes}_{\A} \C})$ by definition, see \cite[p. 298]{BL1}. Therefore $\overline{\B \widehat{\otimes}_{\A} \C}$ is an affine formal model in $B \widehat{\otimes}_A C$ by \cite[Corollary 4.6]{BL1}. It contains the $\L$-stable subalgebra generated by the images of $\B \widehat{\otimes} 1$ and $1 \widehat{\otimes} \C$ as a dense subspace, and therefore is $\L$-stable because every $K$-linear derivation of the $K$-affinoid algebra $B \widehat{\otimes}_A C$ is automatically continuous.
\end{proof}

We will denote the full subcategory of $X_w$ consisting of the $\L$-admissible affinoid subdomains by $X_w(\L)$. Part (d) of the Lemma implies the following
\begin{cor}
Let $X$ be a reduced $K$-affinoid variety, and let $\L$ be an $(\R,\A)$-Lie algebra for some affine formal model $\A$ in $\O(X)$. Then $X_w(\L)$ is stable under finite intersections.
\end{cor}
We define an \emph{$\L$-admissible covering} of an $\L$-admissible affinoid subdomain of $X$ to be a finite covering by objects in $X_w(\L)$. The Corollary now shows that $X_w(\L)$ is a $G$-topology on $X$ in the sense of \cite[Definition 9.1.1/1]{BGR}.

It follows from the universal property of the enveloping algebras of Lie--Rinehart algebras and Proposition \ref{AdmRalg}(c) that $\hsULK$ is a functor from $X_w(\L)$ to $K$-Banach algebras. We have thus defined a presheaf $\hsULK$ on $X_w(\L)$. It behaves well with respect to restriction, in the following sense.

\begin{prop}Let $X$ be a reduced $K$-affinoid variety, and let $\L$ be an $(\R,\A)$-Lie algebra for some affine formal model $\A$ in $\O(X)$. Let $Y$ be an $\L$-admissible affinoid subdomain of $X$, let $\B$ be an $\L$-stable affine formal model in $\O(Y)$ and let $\L' = \B \otimes_{\A} \L$. Then there is a natural isomorphism 
\[\hK{\sU(\L')}(Z) \stackrel{\cong}{\longrightarrow} \hsULK(Z)\]
for every $\L'$-admissible affinoid subdomain $Z$ of $Y$.
\end{prop}
\begin{proof} Note that $\L'$ is an $(\R, \B)$-Lie algebra by part (a) of the Lemma, $Z$ is an $\L$-admissible affinoid subdomain of $X$ by part (b) of the Lemma, and the largest $\L'$-stable affine formal model $\C$ in $\O(Z)$ coincides with $\O(Z)^\diamond$ by part (c) of the Lemma. Hence
 $\C \otimes_{\B} \L' \cong \O(Z)^\diamond\otimes_{\A} \L$ and the result follows.
\end{proof}

\subsection{Comparing two affine formal models} \label{Compare}
Let $V$ and $W$ be two $K$-Banach spaces. We say that a $K$-linear map $f : V \to W$ is an \emph{isomorphism} if it is bounded, and has a bounded $K$-linear inverse.

\begin{prop}
Let $\sigma : A \to B$ be a morphism of reduced $K$-affinoid algebras, let $\A$ be an affine formal model in $A$ and let $\B \subset \B'$ be affine formal models in $B$ containing $\sigma(\A)$.
\be
\item Let $Q$ be a flat $\A$-module. Then there is an isomorphism of Banach $B$-modules
\[ \widehat{\B \otimes_{\A} Q} \otimes_\R K \stackrel{\cong}{\longrightarrow} \widehat{\B' \otimes_{\A} Q} \otimes_\R K.\]
\item Suppose that $\sigma$ is \'etale and that $A$ and $B$ are reduced. Let $\L$ be a smooth $(\R,\A)$-Lie algebra, and suppose that $\B$ and $\B'$ are $\L$-stable. Then there is an isomorphism of $K$-Banach algebras
\[ \widehat{ U( \B \otimes_{\A} \L) } \otimes_\R K  \stackrel{\cong}{\longrightarrow} \widehat{ U( \B'\otimes_{\A} \L) } \otimes_\R K.\]
\ee\end{prop}
\begin{proof} (a) $\B'$ is a finitely generated $\B$-module by Lemma \ref{AdmRalg}(a), so $\pi^a \B' \subset \B$ for some integer $a$. Since $Q$ is flat, we obtain $\B$-module embeddings
\[\B \otimes_{\A} Q \hookrightarrow \B' \otimes_{\A} Q \hookrightarrow \frac{1}{\pi^a}\B \otimes_{\A} Q\]
that induce the required isomorphism $\widehat{\B \otimes_A Q} \otimes_R K \stackrel{\cong}{\longrightarrow} \widehat{\B' \otimes_A Q} \otimes_\R K$ after completing and inverting $\pi$.

(b) Since $\B \subset \B'$ are $\L$-stable affine formal models, the natural inclusion $\B \otimes_{\A} \L \to \B' \otimes_{\A} \L$ is a homomorphism of $(\R,\B)$-Lie algebras, which induces an $\R$-algebra homomorphism $U(\B \otimes_{\A} \L) \to U(\B' \otimes_{\A} \L)$, and a $K$-algebra homomorphism 
\[\widehat{ U( \B \otimes_{\A} \L) } \otimes_\R K  \to \widehat{ U( \B'\otimes_{\A} \L) } \otimes_\R K\]
by functoriality. Now $U(\L)$ is a projective (hence flat) $\A$-module by \cite[Theorem 3.1]{Rinehart}, so in view of Proposition \ref{RineTh}, this homomorphism is an isomorphism of Banach $B$-modules by part (a). Hence it is also a $K$-Banach algebra isomorphism. 
\end{proof}

\begin{cor} Suppose that $\L$ is smooth.
\be \item $\hsULK$ is isomorphic to the restriction of $\F_{\O(X)^\circ  \otimes_{\A}U(\L)}$ to $X_w(\L)$. 
\item The presheaf of $K$-Banach algebras $\hsULK$ on $X_w(\L)$ is a sheaf. 
\ee
\end{cor}
\begin{proof} (a) Let $Y$ be an $\L$-admissible affinoid subdomain of $X$. Then 
\[ \begin{array}{rcl} \hsULK(Y) & \cong & \widehat{ \O(Y)^\diamond \otimes_{\A} U(\L) } \otimes_\R K, \qmb{whereas} \\

\F_{\O(X)^\circ  \otimes_{\A}U(\L)}(Y) &\cong& \widehat{\O(Y)^\circ \otimes_{\A} U(\L)} \otimes_\R K. \end{array} 
\]
Since $\O(Y)^\diamond \subseteq \O(Y)^\circ$, we may apply the Proposition.

(b) This follows from Lemma \ref{SheafFQ} and part (a).
\end{proof}

We will now prove a generalisation of Tate's Acyclicity Theorem for $\hsULK$.

\subsection{Theorem} \label{TateFQ}
Let $X$ be a reduced $K$-affinoid variety and let $Q$ be a flat $\O(X)^\circ$-module. Then every finite affinoid covering $\U$ of $X$ is $\F_Q$-acyclic.
\begin{proof} By \cite[Lemma 8.2.2/2]{BGR} there is a rational covering $\V$ of $X$ which refines $\U$. The restriction of $\F_Q$ to an affinoid subdomain $Y$ of $X$ is $\F_{Q'}$ where $Q' = \O(Y)^\circ \otimes_{\O(X)^\circ} Q$ is a flat $\O(Y)^\circ$-module. Therefore by \cite[Corollary 8.1.4/3]{BGR}, it is enough to prove that $\V$ is $\F_Q$-acyclic. We can therefore assume that $\U$ is already a rational covering, associated to a collection of functions $f_1,\ldots, f_n \in \O(X)$ without common zero:
\[\U = \{X_1,\ldots,X_n\}\qmb{where} X_i = X \left( \frac{f_1}{f_i}, \ldots, \frac{f_n}{f_i} \right).\]
By rescaling the $f_i$ we may assume that $f_i \in \A := \O(X)^\circ$ for all $i$. 

Since $X$ is reduced, $\A$ an admissible $\R$-algebra by Lemma \ref{AdmRalg}(b). The ideal $\mathcal{I}$ in $\A$ generated by the $f_i$ contains a power of $\pi$ since the $f_i$ have no common zero in $\O(X)$, so we may consider its admissible formal blowing-up $f : \mathcal{X}' \to \mathcal{X} := \Spf(\A)$. By \cite[Theorem 4.1]{BL1}, there is an open Zariski covering $\U' = \{\mathcal{X}_1', \ldots, \mathcal{X}_n'\}$ of $\mathcal{X}'$ such that $X_i = \mathcal{X}_{i,\rig}'$ for all $i$. Therefore 
\[\O(\mathcal{X}_{i_1}' \cap \cdots \cap \mathcal{X}_{i_p}')  \otimes_\R K \stackrel{\cong}{\longrightarrow} \O(X_{i_1} \cap \cdots \cap X_{i_p})  \qmb{for all} i_1 < \cdots < i_p\]
by \cite[Corollary 4.6]{BL1}. The sheaf $\O$ has vanishing higher cohomology on each affine Noetherian formal scheme $\mathcal{X}_{i_1}' \cap \mathcal{X}_{i_2}' \cap \cdots \cap \mathcal{X}_{i_p}'$  by \cite[Theorem II.9.7]{Hart} so $\check{H}^i(\U', \O) \cong H^i(\mathcal{X}',\O)$ by \cite[Exercise III.4.11]{Hart} for all $i$. As the map $f : \mathcal{X}' \to \mathcal{X}$ is proper by construction, it follows that all cohomology groups of the augmented \v{C}ech complex $C^\bullet := C^\bullet_{\aug} (\U', \mathcal{O})$ are finitely generated $\A$-modules by \cite[Corollary 3.4.4]{EGAII}. 

Since each of the finitely many $\R$-algebras appearing in $C^\bullet$ is reduced and admissible, it follows from Proposition \ref{Compare}(a) that
\[ C^\bullet_{\aug}(\U,\F_Q) \cong \widehat{C^\bullet\otimes_{\A}  Q} \otimes_\R K.\]
Now $C^\bullet \otimes_\R K = C^\bullet_{\aug}( \U, \mathcal{O})$ by construction and this last complex is exact by Tate's Theorem \cite[Theorem 8.2.1/1]{BGR}. Therefore $C^\bullet$ has \emph{bounded} $\pi$-torsion cohomology, and the same holds for $C^\bullet \otimes_{\A} Q$ because $Q$ is a flat $\A$-module by assumption.  Therefore $\widehat{C^\bullet\otimes_{\A}  Q} \otimes_\R K$ is exact by Lemma \ref{BddTorsionHat}.
\end{proof}

\begin{cor} Let $X$ be a reduced $K$-affinoid variety and let $\L$ be a smooth $(\R,\A)$-Lie algebra for some affine formal model $\A$ in $\O(X)$. Then every  $\L$-admissible covering of $X$ is $\hsULK$-acyclic.
\end{cor}
\begin{proof} The $\O(X)^\circ$-module $Q = \O(X)^\circ \otimes_{\A} U(\L)$ is projective by \cite[Theorem 3.1]{Rinehart}, hence flat, so the covering is $\mathcal{F}_Q$-acyclic. Now apply Corollary \ref{Compare}(a).
\end{proof}

\section{Exactness of localisation} \label{Flatness}
\subsection{One-variable Tate extensions}\label{OneTate}
Let $A$ be a (not necessarily commutative) Banach $K$-algebra. Then the \emph{free Tate algebra in one variable $t$ over $A$} is 
\[ A\langle t\rangle := \left\{\sum_{i=0}^\infty t^i a_i \in A[[t]] : a_i \to 0 \qmb{as} i \to \infty\right\}.\]
Similarly we can define $M\langle t\rangle$ for a Banach $A$-module $M$, and it is readily checked  that $M\langle t\rangle$ is naturally a Banach $A\langle t\rangle$-module.

We will soon need to understand certain torsion submodules of $M\langle t\rangle $. 

\begin{lem} Let $f \in A$.
\be
\item $M\langle t\rangle$ is $(f t - 1)$-torsion-free.
\item If $(t - f) \cdot \left(\sum_{j=0}^\infty t^j m_j\right) = 0$ then $fm_0 = 0$ and $f m_j = m_{j-1}$ for all $j \geq 1$.
\item If $f$ is central in $A$ and $M$ is Noetherian, then $M\langle t\rangle$ is $(t - f)$-torsion-free.
\ee\end{lem}
\begin{proof} (a) If $(ft - 1) \sum_{j=0}^\infty t^j m_j = 0$ then $m_0 = 0$ and $fm_j = m_{j+1}$ for all $j \geq 0$. Hence $m_j = 0$ for all $j \geq 0$ by induction.

(b) This is a direct calculation, similar to part (a).

(c) Since $f$ acts by $A$-linear endomorphisms of $M$ and $M$ is Noetherian, the ascending chain of $A$-submodules $0 \subseteq \ker f \subseteq \ker f^2 \subseteq \cdots$ in $M$ must terminate at $\ker f^r$, say. Suppose that $(t - f) \cdot \left(\sum_{j=0}^\infty t^j m_j\right) = 0$. Then $f^{j+1} m_j = f m_0 = 0$ for all $j \geq 0$ by part (b), so $m_j \in \ker f^{j+1} = \ker f^r$ for $j \geq r$. Hence $0 = f^r m_{i+r} = m_i$ for all $i \geq 0$ and $\sum_{i=0}^\infty t^im_i = 0$.\end{proof}

\subsection{Lifting derivations from \ts{\A} to \ts{\A\langle t\rangle}}\label{LiftingDersToB}
We begin with an elementary result.

\begin{lem}Let $\A$ be a $\pi$-adically complete $\R$-algebra, let $u$ be an $\R$-linear derivation of $\A$ and let $b \in \A\langle t\rangle$. Then $u$ extends uniquely to an $\R$-linear derivation $v$ of $\A\langle t\rangle$ such that $v(t) = b$.
\end{lem}
\begin{proof} There is a unique $\R$-linear derivation $v_0 : \A[t] \to \A\langle t\rangle$ extending $u : \A \to \A$ such that $v_0(t) = b$. Since $\A$ is $\pi$-adically complete, so is $\A\langle t\rangle$, and $v_0 : \A[t] \to \A\langle t\rangle$ is $\pi$-adically continuous, being $\R$-linear. Since $\A[t]$ is dense in $\A\langle t\rangle$, $v_0$ extends uniquely to an $\R$-linear derivation $v$ of $\A\langle t\rangle$.
\end{proof}

\begin{prop} Let $\A$ be an affine formal model in a reduced $K$-affinoid algebra $A$ and let $\L$ be an $(\R, \A)$-Lie algebra with anchor map $\rho : \L \to \Der_\R(\A)$. Write $x\cdot a = \rho(x)(a)$ for $x \in \L$ and $a \in \A$. Let $f \in A$ be such that $\L \cdot f \subset \A$.  Then there are two lifts \[\sigma_1, \sigma_2 : \L \to \Der_\R(\A\langle t\rangle) \] 
of the action of $\L$ on $\A$ to $\A\langle t\rangle$, such that 
\[\sigma_1(x)(t) = x \cdot f \qmb{and} \sigma_2(x)(t) = -t^2 (x \cdot f) \qmb{for all} x \in \L.\]
\end{prop}
\begin{proof} The $\R$-algebra $\A$ is admissible and is therefore $\pi$-adically complete. By the Lemma, for any $x \in \L$ there is a unique $\R$-linear derivation $\sigma_1(x)$ of $\A\langle t\rangle$ such that $\sigma_1(x)(t) = x\cdot f$. The map $\sigma_1 : \L \to \Der_\R(\A\langle t\rangle)$ obtained in this way is $\A$-linear because $\rho$ is $\A$-linear. Let $x,y \in \L$; then
\[ \sigma_1([x,y])(t) = [x,y]\cdot f = x\cdot (y \cdot f) - y\cdot (x\cdot f) = \left(\sigma_1(x)\sigma_1(y) - \sigma_1(y)\sigma_1(x)\right)(t)\]
so the derivation $\sigma_1([x,y]) - [\sigma_1(x),\sigma_1(y)]$ is identically zero on $\A[t]$. Since $\A\langle t\rangle$ is $\pi$-adically complete and $\A[t]$ is dense in $\A\langle t\rangle$, $\sigma_1$ is a Lie homomorphism.

Similarly we can construct an $\A$-linear map $\sigma_2 : \L \to \Der_\R(\A\langle t\rangle)$ extending $\rho$ such that $\sigma_2(x)(t) = -t^2 x\cdot f$ for all $x \in \L$. Let $x,y \in \L$ and write $x \cdot b = \sigma_2(x)(b)$ for $b \in \A\langle t\rangle$. Then $x \cdot (y \cdot t) = x \cdot(-t^2(y\cdot f)) = 2t^3(x \cdot f)(y \cdot f) - t^2 x \cdot (y \cdot f)$. Therefore $x \cdot (y \cdot t) - y \cdot (x \cdot t) = -t^2 [x,y] \cdot f = [x,y] \cdot t$ because $f \in \A$ and $\rho$ is a Lie homomorphism. Hence $\sigma_2$ is also a Lie homomorphism.\end{proof}

\subsection{\ts{\L}-stable affine formal models for Weierstrass and Laurent domains}\label{WeiLau1}
Let $A$ be a reduced $K$-affinoid algebra, let $B = A \langle t \rangle$ for $i=1,2$ and fix $f \in A$. Let $\A$ be an affine formal model for $A$, and choose $a \in \mathbb{N}$ such that $\pi^af \in \A$. Define
\[ u_1 = \pi^a t - \pi^a f \qmb{and} u_2 := \pi^a f t - \pi^a \in \B := \A\langle t\rangle.\]
Let $X := \Sp(A)$ and let $C_i = B/ u_i B$ be the $K$-affinoid algebras corresponding to the Weierstrass and Laurent subdomains 
\[X_1 := X(f) = \Sp(C_1) \qmb{and} X_2 := X(1/f) = \Sp(C_2)\]
 of $X$, respectively. 

Let $\L$ be an $(\R,\A)$-Lie algebra such that $\L \cdot f \subset \A$. Then by Proposition \ref{LiftingDersToB}, the action of $\L$ on $\A$ lifts to $\B$ in two different ways $\sigma_1$ and $\sigma_2$, and $\L_i := \B \otimes_{\A} \L$ becomes an $(\R, \B)$-Lie algebra by Lemma \ref{BaseExt}, with anchor map $1 \otimes \sigma_i$.

\begin{lem} Let $\L$ be an $(\R,\A)$-Lie algebra and let $f \in A$ be a non-zero element such that $\L \cdot f \subset \A$. Then the affinoid subdomains $X_i$ of $X$ are $\L$-admissible.
\end{lem}
\begin{proof}Let $\C_i := \B / u_i \B$. A direct calculation shows that
\[ \sigma_1(x)(u_1) = 0 \qmb{and} \sigma_2(x)(u_2) = -(x\cdot f) t u_2 \qmb{for all} x \in \L.\]
It follows that $u_i \B$ is a $\sigma_i(\L)$-stable ideal of $\B$, and therefore the image $\overline{\C_i}$ of $\C_i$ in $C_i$ is an $\L$-stable affine formal model in $C_i$. Hence $X_i$ is $\L$-admissible.
\end{proof}

\begin{prop} Let $\L$ be a smooth $(\R,\A)$-Lie algebra and let $f\in A$ be a non-zero element such that $\L \cdot f \subset \A$. 
\be 
\item $U(\L_i) / \pi U(\L_i)$ is isomorphic to $\left(U(\L) / \pi U(\L)\right)[t]$ as a $\B$-module.
\item $\hK{U(\L_i)}$ is a flat $\hK{U(\L)}$-module on both sides.
\item There is a short exact sequence 
\[ 0 \to \hK{U(\L_i)} \stackrel{u_i \cdot}{\longrightarrow} \hK{U(\L_i)} \to \hsULK(X_i) \to 0\]
of right $\hK{U(\L_i)}$-modules, and a short exact sequence
\[ 0 \to \hK{U(\L_i)} \stackrel{\cdot u_i}{\longrightarrow} \hK{U(\L_i)} \to \hsULK(X_i) \to 0\]
of left $\hK{U(\L_i)}$-modules.
\ee
\end{prop}
\begin{proof}
(a) By Proposition \ref{RineTh}, there is a $\B$-module isomorphism $U(\L_i) = U(\B \otimes_{\A} \L) \cong \B \otimes_{\A} U(\L)$. It induces $\B$-module isomorphisms
\[ \frac{U(\L_i)}{\pi U(\L_i)} \cong \frac{\B}{\pi \B} \otimes_{\frac{\A}{\pi \A}} \frac{U(\L)}{\pi U(\L)} \cong k[t] \otimes_k \frac{U(\L)}{\pi U(\L)}.\]

(b) The associated graded ring $\gr \hK{U(\L_i)}$ with respect to the $\pi$-adic filtration is $k[t] \otimes_k \gr \hK{U(\L)}$, which is flat over $\gr \hK{U(\L)}$. Now apply \cite[Proposition 1.2]{ST}.

(c) By symmetry, it is sufficient to prove the first statement. By definition, the sequence $0 \to \B \stackrel{u_i\cdot}{\longrightarrow} \B \to \C_i \to 0$ is exact. Tensor it on the right with the flat left $\B$-module $U(\L_i)$ and apply Proposition \ref{RineTh} to get  a short exact sequence of right $U(\L_i)$-modules $0 \to U(\L_i) \stackrel{u_i \cdot}{\longrightarrow} U(\L_i) \to U(\C_i \otimes_{\A} \L) \to 0$. Since $U(\L_i)$ is Noetherian, $\pi$-adic completion is exact on finitely generated $U(\L_i)$-modules by \cite[\S 3.2.3(ii)]{Berth}. Hence $0 \to \hK{U(\L_i)} \stackrel{u_i \cdot}{\longrightarrow} \hK{U(\L_i)} \to \hK{U(\C_i \otimes_{\A} \L)} \to 0$ is exact. Now there are natural isomorphisms
\[ \hK{U(\C_i \otimes_{\A} \L)} \stackrel{\cong}{\longrightarrow} \hK{U(\overline{\C_i} \otimes_{\A} \L)} \stackrel{\cong}{\longrightarrow}\hK{U(C_i^\diamond \otimes_{\A} \L)} = \hsULK(X_i)\]
by Lemma  \ref{TorsUL} and Proposition \ref{Compare}(b). 
\end{proof}

\begin{rmk} It follows from part (c) of the Proposition that the image of $\hsULK(X)$ in $\hsULK(X_1)$ is dense since it also contains the image of $t$.
\end{rmk}

\subsection{Flatness for Weierstrass and Laurent embeddings}\label{WeiLau2}

We keep the notation from the previous subsection, and recall that $B = A \langle t \rangle$.
\begin{lem} Let $M$ be a finitely generated $\hK{U(\L)}$-module. Then there is a natural isomorphism of Banach $B$-modules $\eta_M \colon M\langle t\rangle \stackrel{\cong}{\longrightarrow} \hK{U(\L_i)} \otimes_{\hK{U(\L)}} M$. 

Similarly, if $N$ is a finitely generated right $\hK{U(\L)}$-module there is a natural isomorphism of Banach $B$-modules $\eta_N\colon N\langle t\rangle \stackrel{\cong}{\longrightarrow} N\otimes_{\hK{U(\L)}}\hK{U(\L_i)}$. 
\end{lem}

\begin{proof} Choose a finitely generated $\h{U(\L)}$-submodule $\M$ in $M$ which generates $M$ as a $K$-vector space. Then
\[\hK{U(\L_i)} \otimes_{\hK{U(\L)}} M \cong \left(\h{U(\L_i)} \otimes_{\h{U(\L)}} \M\right) \otimes_\R K.\]
The finitely generated $\h{U(\L_i)}$-module $\h{U(\L_i)} \otimes_{\h{U(\L)}} \M$ is $\pi$-adically complete by \cite[\S 3.2.3(v)]{Berth} because $\h{U(\L_i)}$ is Noetherian. Therefore, for any sequence of elements $m_j \in M$ tending to zero, the series $\sum_{j=0}^\infty t^j \otimes m_j$ converges to a unique element $\eta_M\left(\sum_{j=0}^\infty t^j m_j \right)$ in $\hK{U(\L_i)} \otimes_{\hK{U(\L)}} M$. Because $t$ commutes with $A$, it is straightforward to see that $\eta_M$ is $B$-linear. It follows from Proposition \ref{WeiLau1}(a) that $\eta_{\hK{U(\L)}}$ is an isomorphism. We may now view $\eta$ as a natural transformation between two right exact functors and use the Five Lemma to conclude that $\eta_{\M}$ is always an isomorphism. The proof of the right module version is similar. 
\end{proof}

\begin{thm} Let $X$ be a reduced $K$-affinoid variety and let $f \in \O(X)$ be non-zero. Let $\A$ be an affine formal model in $\O(X)$ and let $\L$ be a smooth $(\R,\A)$-Lie algebra such that $\L \cdot f \subset \A$. Let $X_1 = X(f)$ and $X_2 = X(1/f)$. Then $\hsULK(X_i)$ is a flat $\hsULK(X)$-module on both sides for $i=1$ and $i=2$.
\end{thm}
\begin{proof} We know that $T_i := \hK{U(\L_i)}$ is a flat right $S := \hK{U(\L)}$-module, and that $\hsULK(X_i) \cong T_i / u_i T_i$ as a right $T_i$-module by Proposition \ref{WeiLau1}. Let $M$ be a finitely generated $S$-module. By Lemmas \ref{FlatMethod} and \ref{WeiLau2}, to prove that $T_i$ is a flat right $S$-module it will be enough to show that the $B$-module $M\langle t \rangle$ is $u_i$-torsion-free. The case $i = 2$ follows immediately from Lemma \ref{OneTate}(a) and $u_1 = \pi^a(t - f)$, so we just have to show that $M \langle t \rangle$ is $(t - f)$-torsion-free.

Suppose now that the element $\sum_{j=0}^\infty t^jm_j \in M \langle t \rangle$ is killed by $t-f$. Then setting $m_{-1} := 0$, we have the equations $fm_j=m_{j-1}$ for all $j \geq 0$ from Lemma \ref{OneTate}(b), and $\lim\limits_{j\to\infty}m_j=0$. We consider the $S$-submodule $N$ of $M$ generated by the $m_j$. Since $M$ is Noetherian, $N$ must be generated by $m_0, \ldots, m_d$ for some $d \geq 0$, say. 

Let $\M$ be a finitely generated $\cS := \h{U(\L)}$-submodule of $M$ which generates $M$ as a $K$-vector space, and let $\N := \sum_{i=0}^d \cS m_i$. Since $\cS$ is Noetherian, $\M \cap N$ is a finitely generated $\cS$-submodule of $N$ which generates $N$ as a $K$-vector space, so the $\cS$-modules $\M \cap N$ and $\N$ contain $\pi$-power multiples of each other. So for all $n \geq 0$ we can find $j_n \geq 0$ such that $m_j \in \pi^n \N$ for all $j \geq j_n$, because $\lim\limits_{j \to \infty} m_j = 0$. 

Since $U(\L)$ is generated by $\A + \L$ as an $\R$-algebra and $[f, \A + \L] \subseteq \L \cdot f \subset \A$  we see that $[f, U(\L)] \subseteq U(\L)$ and consequently $[f,\cS] \subseteq \cS$. Because 
\[ f \sum_{j=0}^d s_jm_j = \sum_{j=0}^d [f,s_j] m_j + s_jm_{j-1} \in \N \qmb{for all} s_0,\ldots, s_d \in \cS\]
we see that $f^i \N \subseteq \N$ for all $i \geq 0$. Therefore for any $j, n \geq 0$ we have
\[ m_j = f^{j_n} m_{j + j_n} \in f^{j_n}\pi^n \N \subseteq \pi^n\N.\]
Hence $m_j \in \bigcap_{n=0}^\infty \pi^n \N = 0$ for all $j \geq 0$ and $\sum_{j=0}^\infty t^jm_j = 0$, so $T_i$ is a flat right $S$-module as claimed. The same argument for finitely generated right $S$-modules $M$ also shows that $T_i$ is a flat left $S$-module. \end{proof}

\subsection{\ts{\L}-accessible rational subdomains}\label{Laccess} ${ }$ \\
Until the end of Section \ref{Flatness}, we will fix the following notation. 
\begin{itemize} 
\item $X$ is a reduced $K$-affinoid variety,
\item $\A$ is  an affine formal model in $\O(X)$,
\item $\L$ is a coherent $(\R,\A)$-Lie algebra,
\item $\cS := \hsULK$.
\end{itemize}

\noindent We would like to prove that every $\L$-admissible \'etale morphism of affinoids $Y \to X$ has the property that $\cS(Y)$ is a flat right and left $\cS(X)$-module. Unfortunately we cannot do this at the moment and we introduce a new notion, that of \emph{$\L$-accessibility}, as a consequence of this.

\begin{defn} 
\be
\item Let $Y \subset X$ be a rational subdomain. If it is the identity map, we say that it is \emph{$\L$-accessible in $0$ steps}. Inductively, if $n \geq 1$ then we say that it is \emph{$\L$-accessible in $n$ steps} if  there exists a chain $Y \subset Z \subset X$, such that
\begin{itemize}
\item $Z \to X$ is $\L$-accessible in $(n-1)$ steps,
\item $Y = Z(f)$ or $Z(1/f)$ for some non-zero $f \in \O(Z)$,
\item there is an $\L$-stable affine formal model $\C \subseteq \O(Z)$ such that $\L \cdot f \subseteq \C$.
\end{itemize}
\item A rational subdomain $Y \subset X$ is said to be \emph{$\L$-accessible} if it is $\L$-accessible in $n$ steps for some $n \in \mathbb{N}$.
\ee
\end{defn}

\begin{prop} Let $Y \subset X$ be an $\L$-accessible rational subdomain. Then it is $\L$-admissible, and $\cS(Y)$ is a flat $\cS(X)$-module on both sides whenever $\L$ is smooth.
\end{prop}
\begin{proof} Assume that $Y \subset X$ is $\L$-accessible in $n$ steps, and proceed by induction on $n$. The statement is vacuous when $n = 0$, so assume $n \geq 1$. Choose a chain $Y \subset Z \subset X$ where $Z \to X$ is $\L$-accessible in $n-1$ steps, assume that $Y = Z(f)$ or $Z(1/f)$ for some non-zero $f \in \O(Z)$ and let $\C \subseteq \O(Z)$ be an $\L$-stable affine formal model such that $\L \cdot f \subset \C$. Then $\L' := \C \otimes_{\A} \L$ is an $(R,\C)$-Lie algebra and $\L' \cdot f \subset \C$, so $Y \subset Z$ is $\L'$-admissible by Lemma \ref{WeiLau1}. Since $Z \subset X$ is $\L$-admissible by induction, $Y \subset X$ is also $\L$-admissible by Lemma \ref{FibreProds}(b). 

Now suppose that $\L$ is smooth. Then $\L'$ is also smooth, so $\hK{\sU(\L')}(Y)$ is a flat $\hK{\sU(\L')}(Z)$-module on both sides by Theorem \ref{WeiLau2}, and Proposition \ref{FibreProds} tells us that $\hK{\sU(\L')}(Y) \cong \cS(Y)$ and $\hK{\sU(\L')}(Z) \cong \cS(Z)$. Since $\cS(Z)$ is a flat  $\cS(X)$-module on both sides by induction, $\cS(Y)$ is also a flat $\cS(X)$-module on both sides.
\end{proof}

\subsection{Proposition}\label{FibreProdLaccess} Let $Y$ be a rational subdomain of $X$ which is $\L$-accessible in $n$ steps.
\be
\item Let $U$ be an $\L$-admissible affinoid subdomain of $X$, and let $\B$ be an $\L$-stable affine formal model in $U$. Then $U \cap Y$ is a rational subdomain of $U$ which is $\L' := \B \otimes_{\A} \L$-accessible in $n$ steps.
\item Let $\B$ be an $\L$-stable affine formal model in $\O(Y)$, and let $Z$ be a rational subdomain of $Y$ which is $\L' := \B \otimes_{\A} \L$-accessible in $m$ steps. Then $Z$ is a rational subdomain of $X$ which is  $\L$-accessible in $(n+m)$ steps.
\ee
\begin{proof} (a) Proceed by induction on $n$, and suppose that $n \geq 1$ as the case when $n = 0$ is trivial. We have a commutative pullback diagram
\[ \xymatrix{ Y \ar[r] & Z \ar[r] & X \\ 
U \cap Y \ar[u]\ar[r] & U \cap Z \ar[u]\ar[r] & U \ar[u]}\]
where $Z \to X$ is $\L$-accessible in $(n-1)$ steps, $Y = Z(f)$ or $Z(1/f)$ for some $f \in \O(Z)$ and $\L \cdot f \subseteq \C$ for some $\L$-stable affine formal model $\C$ in $\O(Z)$. Let $g = 1 \otimes f$ be the image of $f$ in $\O(U \cap Z)$. Then $U \cap Z \to U$ is $\L'$-accessible in $(n-1)$-steps by induction, $U \cap Y$ is either $(U \cap Z)(g)$ or $(U \cap Z)(1/g)$, and $\L' \cdot g \subseteq \overline{ \B \widehat{\otimes}_{\A} \C }$ which is an $\L$-stable affine formal model in $\O(U \cap Z)$ by the proof of Lemma \ref{FibreProds}(d). Therefore $U \cap Y \to U$ is $\L$-accessible in $n$ steps.

(b) Proceed by induction on $m$, and assume that $m \geq 1$ as the case when $m = 0$ is trivial. Choose $Z \to W \to Y$ with $W \to Y$ being $\L'$-accessible in $(m-1)$ steps, and $Z = W(f)$ or $W(1/f)$ for some $f \in \O(W)$, and $\L' \cdot f \subseteq \C$ for some $\L$-stable affine formal model $\C$ in $\O(W)$. Then $W \to X$ is $\L$-accessible in $n + m-1$ steps by induction, and $\L \cdot f \subseteq \C$, so $Z \to X$ is $\L$-accessible in $n + m$ steps.\end{proof}

\begin{cor}  Let $Y \to X$ and $U \to X$ be two $\L$-accessible rational subdomains. Then $U \cap Y \to X$ is also an $\L$-accessible rational subdomain.
\end{cor}
\begin{proof} Choose an $\L$-stable affine formal model $\B$ in $\O(U)$. By part (a) of the Proposition, $U \cap Y \to U$ is a rational subdomain which is $\L' := \B \otimes_{\A} \L$-accessible. Since $U \to X$ is also $\L$-accessible, part (b) of the Proposition (applied to $U \cap Y \to U$) gives that $U \cap Y$ is an $\L$-accessible rational subdomain in $X$.
\end{proof}

\subsection{\ts{\L}-accessible affinoid subdomains}\label{LaccessAff}
Recall that by the Gerritzen-Grauert Theorem \cite[Theorem 4.10.4]{FvdPut}, every affinoid subdomain $Y$ of an affinoid $K$-variety $X$ is actually the union of finitely many rational subdomains in $X$. In view of this fact, we make the following

\begin{defn}\hfill
\be\item An affinoid subdomain $Y$ of $X$ is said to be \emph{$\L$-accessible} if it is $\L$-admissible and there exists a finite covering $Y = \bigcup_{j=1}^r X_j$ where each $X_j$ is an $\L$-accessible rational subdomain of $X$.
\item A finite affinoid covering $\{X_j\}$ of $X$ is said to be \emph{$\L$-accessible} if each $X_j$ is an $\L$-accessible affinoid subdomain of $X$.
\ee\end{defn}

It follows from Proposition \ref{Laccess} that every $\L$-accessible rational subdomain is $\L$-admissible, and is therefore also an $\L$-accessible affinoid subdomain.
\begin{lem}
\be\item The intersection of finitely many $\L$-accessible affinoid subdomains is again an $\L$-accessible affinoid subdomain.
\item If $Z \subset Y$ are $\L$-accessible affinoid subdomains of $X$ and $\B$ is an $\L$-stable affine formal model in $\O(Y)$, then $Z$ is an $\L':=\B\otimes_\A \L$-accessible affinoid subdomain of $Y$.
\ee
\end{lem}
\begin{proof}(a) This follows from Corollary \ref{FibreProdLaccess} together with Corollary \ref{FibreProds}.

(b) Let $\{Z_1,\ldots,Z_n\}$ be a covering of $Z$ by $\L$-accessible rational subdomains of $X$. By Proposition \ref{FibreProdLaccess}(a),  each $Z_i$ is an $\L'$-accessible rational subdomain of $Y$.
\end{proof}
\subsection{Theorem} \label{AffLaccessCovers} Suppose that $\L$ is smooth.
\be
\item Let $Y \subset X$ be an $\L$-accessible affinoid subdomain. \\ Then $\cS(Y)$ is a flat $\cS(X)$-module on both sides.
\item Let $\X = \{X_1, \ldots, X_m\}$ be an $\L$-accessible covering of $X$. \\ Then $\bigoplus_{i=1}^m \cS(X_j)$
is a faithfully flat $\cS(X)$-module on both sides.
\ee\begin{proof} (a) By definition, there is a finite covering $\V = \{X_1,\ldots,X_m\}$ of $Y$ by $\L$-accessible rational subdomains $X_j$. Every finite intersection of these subdomains is $\L$-accessible by Lemma \ref{LaccessAff}(a), so every ring appearing in $C^\bullet(\V, \cS)$ is flat as a $\cS(X)$-module on both sides by Proposition \ref{Laccess}. 

The augmented \v{C}ech complex $C^\bullet_{\aug}(\V, \cS)$ is acyclic by Corollary \ref{TateFQ}. A long exact sequence of $\Tor$ groups shows that the kernel of a surjection between two flat modules is again flat. By an induction starting with the last term, the kernel of every differential in this complex is a flat $\cS(X)$-module on both sides. In particular, $\cS(Y)$ is flat as $\cS(X)$-module on both sides.

(b) By part (a), $\bigoplus_{i=1}^m \cS(X_j)$ is a flat right $\cS(X)$-module. By Lemma \ref{LaccessAff}(a) and part (a), each term in the complex $C^\bullet_{\aug}(\X, \cS)$ is a flat right $\cS(X)$-module. Since it is acyclic by Corollary \ref{TateFQ} we may view it as a flat resolution of the zero module. Let $N$ be a left $\cS(X)$-module. By \cite[Lemma 3.2.8]{Wei1995}, $C^\bullet_{\aug}(\X, \cS) \otimes_{\cS(X)} N$ computes $\Tor^{\cS(X)}_{\bullet}(0, N)$ and is therefore acyclic. So $N$ embeds into $\oplus_{j=1}^m \cS(X_j) \otimes_{\cS(X)} N$ and hence $\oplus_{j=1}^m\cS(X_j)$ is a faithfully flat right $\cS(X)$-module. The same proof shows that it is also a faithfully flat left $\cS(X)$-module.
\end{proof}

\section{Kiehl's Theorem for coherent \ts{{\hK{\sU(\L)}}}-modules} \label{CohMod}
In this section we will make the following standing assumptions:
\begin{itemize} 
\item $X$ is a reduced $K$-affinoid variety,
\item $\A$ is  an affine formal model in $\O(X)$,
\item $\L$ is a smooth $(\R,\A)$-Lie algebra,
\item $\cS := \hsULK$.
\end{itemize}
It follows from Lemma \ref{LaccessAff}(a) that the $\L$-accessible affinoid subdomains of $X$ together with the $\L$-accessible coverings form a $G$-topology on $X$. We will denote this $G$-topology by $X_{ac}(\L)$. Thus we have at our disposal four different $G$-topologies on $X$, represented on the level of objects as follows:
\[ X_{ac}(\L) \quad\subset \quad X_w(\L) \quad\subset\quad X_w \quad\subset\quad X_{\rig}.\]
  
\subsection{Localisation}\label{LaccLoc}
Suppose that $Y$ is an $\L$-admissible affinoid subdomain of $X$. For every finitely generated $\cS(X)$-module $M$,  we can define a presheaf of $\cS$-modules $\Loc(M)$ on $X_w(\L)$ by setting 
\[ \Loc(M)(Y):=\cS(Y)\otimes_{\cS(X)}M.\]
Similarly, for every finitely generated right $\cS(X)$-module $M$, we can define a presheaf of right $\cS$-modules $\Loc(M)$ on $X_w(\L)$ by setting
\[\Loc(M)(Y) := M \otimes_{\cS(X)} \cS(Y).\]
We will frequently use the fact that $\cS(Z)$ is a flat $\cS(Y)$-module on both sides whenever $Z \subset Y$ is an inclusion of $\L$-accessible affinoids of $X$ --- this follows from  Theorem \ref{AffLaccessCovers}(a).

\begin{prop} $\Loc$ is a full exact embedding of abelian categories from the category of finitely generated $\cS(X)$-modules (respectively, right $\cS(X)$-modules) to the category of sheaves of $\cS$-modules (respectively, right $\cS$-modules) on $X_{ac}(\L)$ with vanishing higher \v{C}ech cohomology groups.
\end{prop}
\begin{proof} First we prove that if $M$ is any finitely generated $\cS(X)$-module then every $X_{ac}(\L)$-covering $\U=\{U_1,\ldots,U_n\}$ of every $\L$-accessible affinoid subdomain $Y$ of $X$ is $\Loc(M)$-acyclic. In particular this will demonstrate that $\Loc(M)$ is a sheaf on $X_{ac}(\L)$ with vanishing higher \v{C}ech cohomology groups. 

Let $\B$ be an $\L$-stable affine formal model in $Y$ and let $\L'=\B\otimes_\A L$. By Lemma \ref{LaccessAff}(b) we can view $\U$ as a covering of $Y$ in $Y_{ac}(\L')$. Then $\U$ is $\cS$-acyclic by Corollary \ref{TateFQ}. But every term in the \v{C}ech complex $C^\bullet_{\aug}(\U, \cS)$ is a flat right $\cS(Y)=\hK{U(\L')}$ module by Theorem \ref{AffLaccessCovers}(a). Therefore 
\[C^{\bullet}_{\aug}(\U,\Loc(M)) \cong C^\bullet_{\aug}(\U, \cS) \otimes_{\cS(Y)} \Loc(M)(Y)\]
is also acyclic as claimed. Now suppose that $f\colon M\to N$ is a morphism of finitely generated $\cS(X)$-modules. By the universal property of tensor product, for each $Y$ in $X_{ac}(\L)$ there is a unique morphism of $\cS(Y)$-modules $\Loc(M)(Y)\to \Loc(N)(Y)$ making the diagram  \[\begin{CD}
M @>>> N\\
@VVV @VVV\\
\cS(Y)\otimes_{\cS(X)} M @>>> \cS(Y)\otimes_{\cS(X)} N \end{CD}\] commute. It is now easy to see that $\Loc$ is a full functor as claimed.

Finally, suppose that $0\to M_1\to M_2\to M_3\to 0$ is an exact sequence of finitely generated $\cS(X)$-modules. Since $\cS(Y)$ is a flat right $\cS(X)$-module for each $Y \in X_{ac}(\L)$, each sequence \[0\to \Loc(M_1)(Y)\to\Loc(M_2)(Y)\to \Loc(M_3)(Y)\to 0 \] is exact. This suffices to see that $\Loc$ is exact. 

The case of right modules is almost identical.
\end{proof}

\subsection{\ts{\U}-coherent modules}\label{UcohMod}

Following \cite[\S 9.4.3]{BGR}, we say that an $\cS$-module $\M$ is \emph{coherent} if there is an $X_{ac}(\L)$-covering $\U=\{U_1,\ldots,U_n\}$ of $X$ such that, for each $1\le i\le n$, $\M|_{U_i}$ may be presented by an exact sequence of the form \[ \cS^r|_{U_i}\to \cS^s|_{U_i}\to \M_{|U_i}\to 0.\] Using Proposition \ref{LaccLoc}, we note that in this situation, if we choose $\L$-stable affine formal models $\B_i$ in $U_i$ and write $\L_i=\B_i\otimes_\A \L$, we may view the morphism $\cS^r|_{U_i}\to \cS^s|_{U_i}$ as $\Loc(f_i)$ for some $\cS(U_i)$-linear map $f_i\colon \cS(U_i)^r\to \cS(U_i)^s$. Writing $M_i$ for the cokernel of $f_i$ and applying Proposition \ref{LaccLoc} again we see that there is an isomorphism $\M_{|U_i}\cong \Loc(M_i)$ as $\cS|_{U_i}$-modules since both arise as the cokernel of $\Loc(f_i)$. 

Since each ring $\cS(U_i)$ is left Noetherian and so every finitely generated $\cS(U_i)$-module is finitely presented, it follows from the discussion above that an $\cS$-module $\M$ is coherent precisely if there is an $X_{ac}(\L)$-covering $\{U_1,\ldots,U_n\}$ of $X$ such that, for each $1\le i\le n$, $\M|_{U_i}$ is isomorphic to $\Loc(M_i)$ for some finitely generated $\cS(U_i)$-module $M_i$.

\begin{defn} Given an $X_{ac}(\L)$-covering $\U=\{U_1,\ldots,U_n\}$ of $X$, we say that an $\cS$-module  (respectively, right $\cS$-module) $\M$ is \emph{$\U$-coherent} if for each $1\le i\le n$ there is a finitely generated $\cS(U_i)$-module (respectively, right $\cS(U_i)$-module) $M_i$ such that $\M|_{U_i}$ is isomorphic to $\Loc(M_i)$ as a sheaf of $\cS|_{U_i}$-modules (respectively, right $\cS|_{U_i}$-modules).
\end{defn}

\begin{prop} Let $\U$ be an admissible covering of $X$, and suppose that $\alpha\colon \M\to \N$ is a morphism of $\U$-coherent left or right $\cS$-modules. Then $\ker \alpha, \coker \alpha$ and $\im \alpha$ are each $\U$-coherent. 
\end{prop}

\begin{proof} We compute using Proposition \ref{LaccLoc} that $(\ker \alpha)|_{U_i}\cong \Loc(\ker \alpha(U_i))$, that $(\coker \alpha)|_{U_i}\cong \Loc(\coker \alpha(U_i))$ and that $\im \alpha|_{U_i}=\Loc(\im \alpha(U_i))$. 
\end{proof}

\subsection{Coverings of the form \ts{X=X(f)\cup X(1/f)}}\label{CartanLemma} We generalise some technical results from \cite[\S 4.5]{FvdPut} to our non-commutative setting. This involves making appropriate changes to the material presented in \cite[\S 4.5]{FvdPut}, but we repeat these proofs here nevertheless.  Note that it is incorrectly asserted in the proof of \cite[Lemma 4.5.4]{FvdPut} that $s_2$ has dense image; in fact it is the map $s_1$ that has dense image.

First, we suppose that $f\in \O(X)$ is such that $\L\cdot f\subset \A$.  Then 
\[X_1:=X(f), \quad X_2:=X(1/f) \qmb{and} X_3:=X(f)\cap X(1/f)\]
are all $\L$-accessible. We write $s_i\colon \cS(X_i)\to \cS(X_3)$ for the canonical restriction maps ($i=1,2$).
We define the \emph{norm} $||M||$ of a matrix $M$ with entries in a $K$-Banach algebra to be the supremum of the norms of the entries of $M$.

\begin{lem} There is a constant $c>0$ such that every matrix $M\in M_n(\cS(X_3))$ with $||M-I||<c$ can be written as a product $M=s_1(Q_1)^{-1}\cdot s_2(Q_2)^{-1}$ for some $Q_i \in \GL_n(\cS(X_i))$.
\end{lem}

\begin{proof} By Corollary \ref{TateFQ}, the bounded $K$-linear map 
\[ s_1 - s_2 : \cS(X_1)\oplus \cS(X_2)\to \cS(X_3)\]
is surjective. So, by Banach's Open Mapping Theorem there is a constant $0 < d < 1$ such that if $N$ is any $n\times n$ matrix with entries in $\cS(X_3)$ we can find $N_1 \in M_n(\cS(X_1))$ and  $N_2 \in M_n(\cS(X_2))$ such that
\[N=s_1(N_1)-s_2(N_2) \qmb{and} d\cdot \sup(||N_1||,||N_2||)\leq ||N||.\]
We define $c := d^3$.  Suppose now that $M \in \GL_n(\cS(X_3))$ satisfies $||M - I|| < c$ and let $A_1 = M - I$. We can then find $B_{1i} \in M_n(\cS(X_i))$ of norm at most $d^2$ such that $A_1=s_1(B_{11})+s_2(B_{12})$. Then 
\begin{eqnarray*} 
 A_2  &:=& (I-s_1(B_{11}))(I+A_1)(I-s_2(B_{12}))-I \\
  &=& s_1(B_{11}) s_2(B_{12}) - s_1(B_{11}) A_1 - A_1 s_2(B_{12}) - s_1(B_{11}) \cdot A_1 \cdot s_2(B_{12})
\end{eqnarray*}
is a matrix with coefficients in $\O(X_3)$ and has norm at most $d^4$.

Inductively, we can find sequences $A_m,B_{m1},B_{m2}$ of matrices with coefficients in $\cS(X_3), \cS(X_1)$, $\cS(X_2)$ and norms bounded by $d^{m+1},d^m$ and $d^m$ respectively such that $A_m=s_1(B_{m1})+s_2(B_{m2})$ and \[ A_{m+1} :=(I-s_1(B_{m1}))(I+A_m)(I-s_2(B_{m2}))-I.\] Because $d^m \to 0$ as $m\to \infty$, the limit
\[Q_i := \lim_{m\to \infty} (1-B_{mi})\cdots(1-B_{1i})\]
exists in $M_n(\cS(X_i))$ and $Q_i \in \GL_n(\cS(X_i))$ for $i=1,2$. By construction, 
\[s_1(Q_1) \cdot M \cdot s_2(Q_2)=I\] 
so $M=s_1(Q_1)^{-1}\cdot s_2(Q_2)^{-1}$ as claimed.
\end{proof}

\subsection{Theorem}\label{SpecialLaurent} Suppose that $\N$ is an $\{X_1,X_2\}$-coherent sheaf of $\cS$-modules. Then the canonical $\cS(X_i)$-linear maps $\cS(X_i)\otimes_{\cS(X)}\N(X)\to \N(X_i)$ are surjective for $i=1$ and $i=2$.

Similarly, if $\N$ is an $\{X_1,X_2\}$-coherent sheaf of right $\cS$-modules, then the canonical $\cS(X_i)$-linear maps $\N(X) \otimes_{\cS(X)} \cS(X_i) \to \N(X_i)$ are surjective for $i=1$ and $i=2$.
\begin{proof} We first deal with the case of left $\cS$-modules. Let us identify $\N(X_3)$ with $\cS(X_3) \otimes_{\cS(X_1)} \N(X_1)$ and with $\cS(X_3) \otimes_{\cS(X_2)} \N(X_2)$. Suppose that $a_1,\ldots,a_n$ generate $\N(X_1)$ as a $\cS(X_1)$-module and $b_1,\ldots,b_n$ generate $\N(X_2)$ as a $\cS(X_2)$-module. Then the sets $\{1\otimes a_1,\ldots,1\otimes a_n\}$ and $\{1\otimes b_1,\ldots,1\otimes b_n\}$ each generate $\N(X_3)$ as a $\cS(X_3)$-module. 

Consider $\N(X_3)^n$ as a left module over the $n\times n$ matrix ring $M_n(\cS(X_3))$ and let $\mathbf{a}, \mathbf{b} \in \N(X_3)^n$ be the column vectors whose $j$th entries are $1 \otimes a_j$ and $1 \otimes b_j$, respectively. Then we may find non-zero $U, V \in M_n(\cS(X_3))$ such that
\[ \mathbf{a} = U \mathbf{b} \qmb{and} \mathbf{b} = V \mathbf{a}.\]
Let $c$ denote the constant from Lemma \ref{CartanLemma}. Since the image of $s_1\colon \cS(X_1)\to \cS(X_3)$ is dense by Remark \ref{WeiLau1}, we can find $V' \in M_n(\cS(X_1))$ such that 
\[ || s_1(V') - V || < c / || U || .\]
Therefore $||(s_1(V')-V)U||<c$, and by Lemma \ref{CartanLemma}, we can find $Q_i\in \GL_n(\cS(X_i))$ for $i=1,2$ such that
\[I+(s_1(V')-V)U=s_1(Q_1)^{-1}s_2(Q_2)^{-1}.\]
Applying this matrix identity to the vector $\mathbf{b} \in \N(X_3)^n$ we obtain
\[s_1(Q_1 V') \mathbf{a} = s_2(Q_2^{-1}) \mathbf{b}.\]
Writing $a_i'=\sum_{j=1}^n (Q_1V')_{ij}a_j \in \N(X_1)$ and $b_i'=\sum_{j=1}^n (Q_2^{-1})_{ij}b_j \in \N(X_2)$, we see that $1\otimes a_i'=1\otimes b_i'$ in $\N(X_3)$ for each $i = 1,\ldots, n$. Since $\N$ is a sheaf, we can find elements $d_1,\ldots,d_n\in \N(X)$ such that the image of $d_i$ in $\N(X_1)$ is $a_i'$ and the image of $d_i$ in $\N(X_2)$ is $b_i'$ for each $i=1,\ldots,n$. Since the matrix $Q_2^{-1}$ is invertible, the elements $b_1',\ldots, b_n'$ generate $\N(X_2)$ as an $\cS(X_2)$-module. Therefore the map $\cS(X_2)\otimes_{\cS(X)}\N(X) \to \N(X_2)$ is surjective.

Now consider an arbitrary element $v\in \N(X_1)$. Since $1\otimes b'_1,\ldots, 1\otimes b_n'$ generate $\N(X_3)$ as a $\cS(X_3)$-module we can write $1\otimes v=\sum_{i=1}^n z_i\otimes b_i'$ for some $z_i\in \cS(X_3)$. The surjectivity of $\cS(X_1)\oplus \cS(X_2)\to \cS(X_3)$ means that we can find $x_i\in \cS(X_1)$ and $y_i\in \cS(X_2)$ such that  $z_i=s_1(x_i)+s_2(y_i)$ for each $i=1,\ldots, n$. Therefore
\[1 \otimes (v - \sum_{i=1}^n x_ia_i') = 1\otimes v- \sum_{i=1}^n s_1(x_i)\otimes a_i'=\sum_{i=1}^n s_2(y_i)\otimes b_i' =  1 \otimes \sum_{i=1}^n y_i b_i'\]
inside $\N(X_3)$, because $1 \otimes a_i' = 1 \otimes b_i'$ for all $i$. Since $\N$ is a sheaf, there is an element $w\in \N(X)$ whose image in $\N(X_1)$ is $v - \sum_{i=1}^n x_ia_i'$ and whose image in $\N(X_2)$ is $\sum_{i=1}^n y_ib_i'$. In particular, $v$ is the image of $1 \otimes w + \sum_{i=1}^n x_i \otimes d_i$ under the map $\cS(X_1)\otimes_{\cS(X)}\N(X)\to \N(X_1)$. Therefore this map is also surjective.

In the case of right $\cS$-modules, again we can find a generating set $\{a_1,\ldots, a_n\}$ for $\N(X_1)$ as a right $\cS(X_1)$-module, and a generating set $\{b_1,\ldots, b_n\}$ for $\N(X_2)$ as a right $\cS(X_2)$-module. Then $\{a_1 \otimes 1, \ldots, a_n \otimes 1\}$ and $\{b_1 \otimes 1, \ldots, b_n \otimes 1\}$ each generate $\N(X_3)$ as a right $\cS(X_3)$-module. We consider $\N(X_3)^n$ as a right module over the $n \times n$ matrix ring $M_n(\cS(X_3))$ and let $\mathbf{a}, \mathbf{b} \in \N(X_3)^n$ be the \emph{row} vectors whose $j$th entries are $a_j \otimes 1$ and $b_j \otimes 1$, respectively. Then we may find non-zero $U, V \in M_n(\cS(X_3))$ such that $\mathbf{a} = \mathbf{b} U$ and $\mathbf{b} = \mathbf{a}V.$ Choose $V' \in M_n(\cS(X_1))$ as above satisfying $||U(s_1(V') - V)|| < c$, and let $T :=  U (s_1(V') - V)$. Then $|| (I + T)^{-1} - I || < c$ also, so by Lemma \ref{CartanLemma}, we can find $Q_i \in \GL_n(\cS(X_i))$ for $i=1,2$ such that $(I + T)^{-1} = s_1(Q_1)^{-1} s_2(Q_2)^{-1}.$ Hence $I + T = s_2(Q_2) s_1(Q_1)$, and applying this matrix identity to the vector $\mathbf{b} \in \N(X_3)^n$ we obtain $\mathbf{a} s_1(V'Q_1^{-1}) = \mathbf{b} s_2(Q_2).$ Therefore the elements $b_j' :=\sum_{i=1}^n b_i (Q_2)_{ij} \in \N(X_2)$ extend to global sections of $\N$ and generate $\N(X_2)$ as a right $\cS(X_2)$-module because the matrix $Q_2$ is invertible. Thus $\N(X) \otimes_{\cS(X)} \cS(X_2) \to \cS(X_1)$ is surjective, and the same argument as in the case of left modules now shows that $\N(X) \otimes_{\cS(X)} \cS(X_1) \to \N(X_1)$ is also surjective.
\end{proof}

\begin{cor} If $\N$ is an $\{X(f),X(1/f)\}$-coherent sheaf of $\cS$-modules then there is a finitely generated $\cS(X)$-module $N$ such that $\Loc(N)\cong \N$. A similar statement holds for an $\{X(f), X(1/f)\}$-coherent sheaf of right $\cS$-modules.
\end{cor}
 
\begin{proof} By symmetry, it will suffice to treat the case of left $\cS$-modules. As before write $X_1=X(f)$ and $X_2=X(1/f)$. By the Theorem, the natural maps $\cS(X_i)\otimes_{\cS(X)}\N(X) \to \N(X_i)$ are surjective for $i=1,2$. Since $\N(X_i)$ is a Noetherian $\cS(X_i)$-module, we can find a finitely generated $\cS(X)$-submodule $M$ of $\N(X)$ such that $\cS(X_i) \otimes_{\cS(X)} M \to \N(X_i)$ is surjective for $i=1,2$. Thus the natural map $\alpha\colon \Loc(M)\to \N$ is surjective since its restrictions to $X_1$ and $X_2$ are both surjective. Since $\Loc(M)$ and $\N$ are both $\{X_1, X_2\}$-coherent, $\ker \alpha$ is also $\{X_1, X_2\}$-coherent by Proposition \ref{UcohMod} so we may find a finitely generated $\cS(X)$-submodule $M'$ of $(\ker \alpha)(X)$ such that $\Loc(M')\to \ker \alpha$ is surjective. Thus $\N$ is isomorphic to the cokernel of $\Loc(M')\to \Loc(M)$. Since $\Loc$ is full, this cokernel is isomorphic to $\Loc(\coker (M'\to M))$ and we are done.
\end{proof}

Here is our non-commutative version of Kiehl's Theorem for sheaves of $\cS$-modules and $\L$-accessible Laurent coverings.

\subsection{Theorem} \label{LaccKiehl}Suppose that $f_1,\ldots,f_n\in \O(X)$ are such that $\L\cdot f_i\subset \A$ for each $i=1,\ldots,n$. Let $\U$ be the Laurent covering $\{X(f_1^{\alpha_1},\ldots,f_n^{\alpha_n})\st \alpha_i\in \{\pm 1\}\}$. Then $\U$ is $\L$-accessible and every $\U$-coherent sheaf $\M$ of left (respectively, right) $\cS$-modules on $X_{ac}(\L)$ is isomorphic to $\Loc(M)$ for some finitely generated left (respectively, right) $\cS(X)$-module $M$.
\begin{proof} The $\L$-accessibility of $\U$ follows from Corollary \ref{FibreProdLaccess}. By symmetry, it is sufficient to treat the case of left $\cS$-modules. We proceed by induction on $n$, the case $n=1$ being Corollary \ref{SpecialLaurent}. 

Suppose that $n> 1$, and that for every family ($X$, $\A$, $\L$) satisfying our standing hypotheses, the result is known for all smaller values of $n$. Suppose also that $f_1,\ldots,f_{n}\in \O(X)$ satisfy the hypotheses of the Proposition and that $\M$ is $\U$-coherent. 

Consider the cover $\V:=\{X(f_{n})(f_1^{\alpha_1},\ldots,f_{n-1}^{\alpha_n})\st \alpha_i\in \{\pm 1\}\}$ of $X(f_n)$. Let $\B$ be an $\L$-stable affine formal model for $X(f_n)$; then $\L'=\B\otimes_\A \L$ is a smooth $(\R, \B)$-Lie algebra. Now $\L'\cdot f_i\subset \B$ for all $i < n$, and since $\M|_{\V}$ is $\V$-coherent the induction hypothesis gives that $\M|_{X(f_n)}$ is isomorphic to $\Loc(M_1)$ for some finitely generated $\cS(X(f_{n}))$-module $M_1$. 

Using an identical argument for $X(1/f_{n})$, $\M|_{X(1/f_{n})}$ is isomorphic to $\Loc(M_2)$ for some finitely generated $\cS(X(1/f_{n}))$-module $M_2$.  Applying Corollary \ref{SpecialLaurent} again completes the proof.
\end{proof}

\section{Fr\'echet--Stein enveloping algebras}\label{FrechEnv}
We assume throughout Section \ref{FrechEnv} that $A$ is a reduced $K$-affinoid algebra and that $L$ is a coherent $(K,A)$-Lie algebra.

\subsection{Lie lattices}\label{LR}

\begin{defn} Let $\A$ be an affine formal model in $A$ and let $\L \subset L$ be an $\A$-submodule.
\be
\item $\L$ is an \emph{$\A$-lattice} if it is finitely generated as an $\A$-module, and $K \L=L$. 
\item $\L$ is a \emph{$\A$-Lie lattice} if in addition it is a sub $(\R,\A)$-Lie algebra of $L$. 
\ee
\end{defn} 

\begin{lem} Let $\L$ be an $\A$-lattice in $L$.
\be
\item If $\L$ is an $\A$-Lie lattice then $\pi^n\L$ is also an $\A$-Lie lattice for all $n\ge 0$.
\item If $\B$ is another affine formal model in $A$ then there is $n \geq 0$ such that 
\[\pi^m \L \cdot \B \subset \B \qmb{for all} m \geq n.\]
\item There is $n\ge 0$ such that $\pi^m \L$ is an $\A$-Lie lattice in $L$ for all $m \geq n$.
\ee
\end{lem}
\begin{proof}
(a) This is clear.

(b) Let $x_1,\ldots,x_d$ generate $\L$ as an $\A$-module, and let $\rho : L \to \Der_K(A)$ be the anchor map. The derivation $\rho(x_i)\colon A\to A$ is bounded for each $i=1,\ldots,d$ --- see the discussion in Section \ref{LiftEtale}. So there is $m_i\geq 0$ such that $\pi^{m_i}\rho(x_i)(\B)\subset \B$. By Lemma \ref{AdmRalg}, we can find $t\geq 0$ such that $\pi^t\A \subseteq \B$. Let $n = t + \max m_i$ and suppose that $m\geq n$. Then
\[ \pi^m \L \cdot \B \subset \sum_{i=1}^d \pi^t \A \hspace{0.1cm}\pi^{m_i}\rho(x_i)(\B) \subset \B.\]
(c) Since $L$ is a $(K,A)$-Lie algebra generated by $x_1,\ldots,x_d$ as an $A$-module, there are $a_{ij}^k\in A$ such that $[x_i,x_j]=\sum_{k=1}^d a_{ij}^k x_k$ for $1\le i,j\le d$. Since $A=K \cdot \A$, there is $s\ge 0$ such that $\pi^s a_{ij}^k\in \A$ for all $i,j$ and $k$. Then for $m\geq s$ we can compute 
\[[\pi^mx_i,\pi^mx_j]\in \sum_{k=1}^d \pi^{2m}a_{ij}^kx_k\in \pi^m \L\]
and hence $\pi^m\L$ is an $R$-Lie algebra for $m \geq s$. Using part (b), we can find $s' \geq 0$ such that $\pi^m \L \cdot \A \subset \A$ for all $m\geq s'$. Now take $n = \max\{s,s'\}$.
\end{proof}

\subsection{Fr\'echet completions of enveloping algebras}\label{DefnOfwUL}

Let $\A$ be an affine formal model in $A$, and let $\L$ be an $\A$-Lie lattice in $L$. We define
\[\w{U(L)}_{\A,\L} := \invlim \hspace{0.1cm} \hK{U(\pi^n \L)}.\]
Being a countable inverse limit of $K$-Banach algebras, $\w{U(L)}_{\A,\L}$ is a Fr\'echet algebra. 

\begin{lem} Let $\A$ be an affine formal model in $A$ and let $\L_1, \L_2$ be two $\A$-Lie lattices in $L$. Then there is a unique continuous $K$-algebra isomorphism
\[ \w{U(L)}_{\A,\L_1} \stackrel{\cong}{\longrightarrow} \w{U(L)}_{\A,\L_2}\]
which restricts to the identity map on $U(L)$.
\end{lem}
\begin{proof}
Since $\L_1 \cap \L_2$ is again an $\A$-Lie lattice in $L$, we may assume without loss of generality that $\L_1 \subset \L_2$. The universal property of $U(-)$ induces $K$-Banach algebra homomorphisms $\hK{U(\pi^n \L_1)} \to \hK{U(\pi^n \L_2)}$ for each $n\geq 0$ and hence a continuous $K$-algebra homomorphism
\[ \alpha : \w{U(L)}_{\A,\L_1} \to \w{U(L)}_{\A,\L_2}.\]
Because $\L_1$ and $\L_2$ are $\A$-lattices in $L$, we can find an integer $s$ such that $\pi^s \L_2 \subseteq \L_1$. This gives $K$-Banach algebra homomorphisms 
$\hK{U(\pi^{n+s}\L_2)} \to \hK{U(\pi^n \L_1)}$ for each $n\geq 0$ and hence a continuous $K$-algebra homomorphism
\[ \beta : \w{U(L)}_{\A,\L_2} \to \w{U(L)}_{\A,\L_1}.\]
It is easy to see that $\alpha$ and $\beta$ are mutually inverse.\end{proof}

Thus $\w{U(L)}_{\A,\L}$ is independent of the choice of $\L$ up to unique isomorphism, and we write $\w{U(L)}_{\A}$ to denote any of these Fr\'echet algebra completions of $U(L)$.

\begin{prop} Let $\A$ and $\B$ be two affine formal models in $A$. Then there is a unique continuous isomorphism
\[\w{U(L)}_{\A} \stackrel{\cong}{\longrightarrow}\w{U(L)}_{\B}\]
which restricts to the identity map on $U(L)$.
\end{prop}
\begin{proof} Choose an $\A$-Lie lattice $\L$ and a $\B$-Lie lattice $\J$ in $L$. Since $K\cdot \A = K\cdot \B = A$, we can an integer $r$ such that $\pi^r \cdot \A \subset \B$. Similarly we can find an integer $s$ such that $\pi^s \cdot \L \subset \J$. 

Let $x_1,\ldots,x_d$ generate $\L$ as an $\A$-module, and let $T$ be the image of $\h{U(\J)}$ inside $\hK{U(\J)}$. The universal property of $U(-)$ induces an $\R$-algebra homomorphism $\theta_0 : U(\pi^s \L) \to \hK{U(\J)}$. Now $U(\pi^s \L)$ is generated as an $\A$-module by the set 
\[\{(\pi^s x_1)^{\alpha_1}\cdots(\pi^s x_d)^{\alpha_d} : \alpha\in\mathbb{N}^d\}.\]
Since $\theta_0$ sends all these elements to $T$ and since $\A \subset \pi^{-r} \B$, we see that the image of $\theta_0$ is contained in $\pi^{-r}T$. Hence $\theta_0$ extends to a $K$-algebra homomorphism
\[\theta_0 : \hK{U(\pi^s \L)} \to \hK{U(\J)}.\]
Applying the same argument to $\pi^{s+n} \cdot \L \subset \pi^n \J$ for each $n\geq 0$, we obtain a compatible sequence of $K$-algebra homomorphisms
\[\theta_n : \hK{U(\pi^{s+n}\L)} \to \hK{U(\pi^n \J)}\]
and hence a continuous $K$-algebra homomorphism
\[\theta_{\A,\B} := \invlim \theta_n: \w{U(L)}_{\A} \to \w{U(L)}_{\B}\]
which restricts to the identity map on $U(L)$. Since $\theta_{\B,\A}\circ\theta_{\A,\B}$ is the identity map on the dense image of $U(L)$ inside $\w{U(L)}_{\A}$, it must be equal to $\id_{\w{U(L)}_{\A}}$. Similarly $\theta_{\A,\B}\circ\theta_{\B,\A} = \id_{\w{U(L)}_{\B}}$.
\end{proof}

\begin{defn} Let $A$ be a reduced $K$-affinoid algebra and let $L$ be a $(K,A)$-Lie algebra which is finitely generated as an $A$-module. The \emph{Fr\'echet completion} of $U(L)$ is
\[ \w{U(L)} := \w{U(L)}_{\A} = \invlim \hspace{0.1cm} \hK{U(\pi^n \L)}\]
for any choice of affine formal model $\A$ in $A$ and $\A$-Lie lattice $\L$ in $L$.
\end{defn}
The above Lemma and Proposition ensure that this definition does not depend on the choice of $\A$ or $\L$, up to unique isomorphism.
\subsection{Functoriality} \label{FuncFS}

Whenever $\sigma : A\to B$ is an \'etale morphism of affinoid algebras, there is a Lie homomorphism $\psi: \Der_K(A) \to \Der_K(B)$ by Lemma \ref{LiftEtale}, and we may view $B\otimes_A L$ as a $(K,B)$-Lie algebra by Corollary \ref{LiftEtale}.

\begin{prop} Let $\sigma : A\to B$ be an \'etale morphism of reduced $K$-affinoid algebras, and let $\varphi : L \to L'$ be a morphism of $(K,A)$-Lie algebras. Then there are unique continuous $K$-algebra homomorphisms
\be\item $\w{U(L)}\to \w{U(B\otimes_A L)}$ extending the natural map $U(L)\to U(B\otimes_A L)$, and
\item $\w{U(L)}\to \w{U(L')}$ extending the natural map $U(L)\to U(L')$.  \ee
\end{prop}
\begin{proof} We will construct an $A^\circ$-Lie lattice $\L$ in $L$ and a $B^\circ$-Lie lattice $\J$ in $B \otimes_A L$ (respectively, an $A^\circ$-Lie lattice $\J$ in $L'$) such that $(\sigma \otimes 1)(\L) \subseteq \J$ (respectively, $\varphi(\L) \subseteq \J$). Then the universal property of $U(-)$ induces continuous $K$-algebra homomorphisms
\[ \hK{U(\pi^m \L)} \to \hK{U(\pi^m \J)}\]
for all $m \geq 0$, and passing to the inverse limit gives the required map $\w{U(L)} \to \w{U(B\otimes_AL)}$ (respectively, $\w{U(L)} \to \w{U(L')}$). In each case uniqueness follows from the density of the image of $U(L)$ in $\w{U(L)}$. 

(a) Choose an $A^\circ$-Lie lattice $\L$ in $L$ and let $\J$ be the image of $B^\circ \otimes_{A^\circ} \L$ in $B \otimes_A L$. Then $\J$ is a $B^\circ$-lattice in $B\otimes_A L$ so by Lemma \ref{LR}(b), $\pi^n \J$ is a $B^\circ$-Lie lattice in $B\otimes_A L$ for some $n \geq 0$, and $(\sigma \otimes 1)(\pi^n \L) \subset \pi^n \J$.

(b) Let $\J$ be an $A^\circ$-Lie lattice in $L'$. Then $\varphi^{-1}(\J)$ generates $L$ as a $K$-vector space and hence contains an $A^\circ$-lattice in $L$. By Lemma \ref{LR}(c), $\varphi^{-1}(\J)$ contains an $A^\circ$-Lie lattice $\L$ in $L$ and $\varphi(\L) \subset \J$.
\end{proof}

\subsection{Fr\'echet-Stein algebras}\label{FsStein}
Following \cite[\S 3]{ST} we say that a $K$-algebra $U$ is \emph{Fr\'echet-Stein} if 
\begin{itemize}
\item there is a tower $U_0 \leftarrow U_1 \leftarrow U_2 \leftarrow \cdots $ of Noetherian $K$-Banach algebras,
\item $U_n$ is a flat right $U_{n+1}$-module for all $n\geq 0$, 
\item $U \cong \invlim U_n$.
\end{itemize}
This definition is designed with a view towards categories of \emph{left} modules. Because we will also need to work with \emph{right} modules in the future, we make this definition more precise by saying that $U$ is \emph{left Fr\'echet-Stein}. If there is a tower $U_0 \leftarrow U_1 \leftarrow U_2 \leftarrow \cdots $ of Noetherian $K$-Banach algebras such that $U \cong \invlim U_n$ and each $U_n$ is a flat \emph{left} $U_{n+1}$-module for all $n \geq 0$, then we say that $U$ is \emph{right Fr\'echet-Stein}. If both conditions are satisfied, then we say that $U$ is \emph{two-sided Fr\'echet-Stein}. 

\begin{thm} Let $A$ be a reduced $K$-affinoid algebra and let $L$ be a coherent $(K,A)$-Lie algebra. Suppose $L$ has a smooth $\A$-Lie lattice $\L$ for some affine formal model $\A$ in $\O(X)$. Then $\w{U(L)}$ is a two-sided Fr\'echet-Stein algebra.
\end{thm}

We start preparing for the proof of this Theorem, which is given below in Section \ref{Onestepflat}. Recall \cite[\S3.5]{AW13}  that a positively filtered $\R$-algebra $U$ is said to be \emph{deformable} if $\gr U$ is flat over $\R$. Its \emph{$n$-th deformation} is by definition its subring
\[U_n:=\sum_{i\ge 0}\pi^{in}F_iU.\]
It follows from \cite[Lemma 3.5]{AW13} that $U_n$ is again a deformable $\R$-algebra, whose filtration is given by  
\[F_jU_n = U_n \cap F_jU = \sum_{i=0}^j \pi^{in} F_iU.\]
We begin by recording some useful general facts on deformable algebras.

\begin{lem} Let $U$ be a deformable $\R$-algebra. Then
\be
\item  $U_1 \cap \pi^tU = \sum_{i \geq t} \pi^i F_iU$ for any $t \geq 0$.
\item $(U_n)_m$ is equal to $U_{m+n}$ for any $n,m \geq 0$.
\ee
\end{lem}
\begin{proof} (a) The $\R$-module $U / F_tU$ is a direct limit of iterated extensions of $\R$-modules of the form $\gr_jU$, each of which is flat by assumption. Hence $U / F_t U$ has no $R$-torsion and consequently $F_tU \cap \pi^t U = \pi^t F_tU$. Since $\sum_{i > t}\pi^i F_iU \subseteq \pi^t U$,
\[ U_1 \cap \pi^t U \subseteq \left(F_tU + \sum_{i > t}\pi^i F_iU\right) \cap \pi^t U \subseteq (F_tU  \cap \pi^t U) + \sum_{i > t}\pi^iF_iU = \sum_{i\geq t}\pi^iF_iU\]
by the modular law, and the reverse inclusion is clear.

(b) $(U_n)_m = \sum_{j\geq 0} \pi^{jm} \sum_{i=0}^j \pi^{in} F_iU = \sum_{i\geq 0} (\sum_{j \geq i} \pi^{jm + in}\R) F_iU = U_{n+m}$.
\end{proof} 

\subsection{The subspace filtration on \ts{U_1}}\label{Iadicpiadic} 
We will need to study the subspace filtration on $U_1$ induced from the $\pi$-adic filtration on $U$ in detail.

\begin{lem} Let $U$ be a deformable $\R$-algebra such that $\gr U$ is commutative. Suppose that $\gr U$ is generated by the symbols of the elements $x_1,\ldots, x_m \in U$ as an algebra over $\gr_0 U$. Let $r_j = \deg x_j$. Then \[ F_iU=F_0U\cdot\left\{x_1^{\alpha_1}\cdots x_m^{\alpha_m}\st \sum \alpha_j r_j\leq i\right\}\]
for each $i \geq 0$.
\end{lem}
\begin{proof} It is sufficient to prove that $F_iU$ is contained in the right hand side, the reverse inclusion being clear. We proceed by induction on $i$, the case $i=0$ being trivial. For every $z\in F_iU$, the image of $z$ in $\gr_iU$ is a $\gr_0U$-linear combination of monomials in the symbols of the $x_j$'s by our assumption. Hence for each $\alpha \in \mathbb{N}^d$ such that $\sum \alpha_jr_j = i$ we can find $\lambda_\alpha \in \gr_0 U = F_0U$ such that 
\[z-\sum\lambda_{\alpha}x^\alpha\in F_{i-1}U.\]
The result follows immediately by applying the inductive hypothesis.
\end{proof}

\begin{prop}
Let $U$ be a deformable $\R$-algebra such that $\gr U$ is a commutative Noetherian graded ring, and let $I := U_1 \cap \pi U$. Then the subspace filtration on $U_1$ arising from the $\pi$-adic filtration on $U$ and the $I$-adic filtration on $U_1$ are topologically equivalent.
\end{prop}
\begin{proof}
Because $\gr U$ is commutative and Noetherian, there are elements $x_1,\ldots, x_m$ in $U$ whose symbols generate $\gr U$ as an algebra over $\gr_0 U$ by \cite[Proposition 10.7]{AMac}. We may assume that each $r_j := \deg x_j$ is positive; then 
\[\pi \in I \quad \mbox{and}\quad \pi^{r_j}x_j \in I \quad\mbox{for all}\quad j \geq 1.\]
Let $r_0 := 1$; it follows from the Lemma that $\pi^i F_iU$ is generated as an $F_0U$-module by all possible elements of the form
\[ (\pi^{r_0})^{\alpha_0} (\pi^{r_1} x_1)^{\alpha_1} \cdots (\pi^{r_m} x_m)^{\alpha_m}\]
where $\alpha_j \in \mathbb{N}$ for all $j=0,\ldots, m$ and $\sum_{j=0}^m \alpha_j r_j = i$. If the integer $t$ is given and $i \geq t \max r_j$, then $\left(\sum_{j=0}^m \alpha_j\right) \max r_j \geq \sum_{j=0}^m \alpha_j r_j = i \geq t \max r_j$, so 
\[(\pi^{r_0})^{\alpha_0} (\pi^{r_1} x_1)^{\alpha_1} \cdots (\pi^{r_m} x_m)^{\alpha_m} \in I^t\]
because $\pi \in I$ and $\pi^{r_j}x_j \in I$ for all $j \geq 1$. Therefore by Lemma \ref{FsStein}(a) we have
\[ U_1 \cap \pi^{t \max r_j} U = \sum_{i \geq t \max r_j} \pi^i F_iU \subseteq I^t \subseteq U_1\cap \pi^t U  \quad\mbox{for all} \quad t \geq 0\]
because $I$ is an $F_0U$-submodule of $U$. \end{proof}

\subsection{\ts{\pi}-adic completions}\label{Onestepflat}
Recall that if $U$ is a deformable $\R$-algebra, then $\h{U_n}:=\invlim U_n/\pi^a U_n$ denotes the $\pi$-adic completion of $U_n$ and that 
\[ \h{U_{n,K}}:=K\otimes_\R \h{U_n}\]
may be equipped with the structure of a $K$-Banach algebra, with unit ball $\h{U_n}$. Since $U_0 = U$, we will abbreviate $\h{U_{0,K}}$ to $\h{U_K}$.

\begin{thm}
Let $U$ be a deformable $\R$-algebra such that $\gr U$ is a commutative Noetherian ring. Then $\h{U_K}$ is a flat  $\h{U_{1,K}}$ module on both sides.
\end{thm}
\begin{proof}
In this proof, "flat module" will mean "flat module on both sides".  Since $\h{U_{1,K}} = \h{U_1} \otimes_\R K$, it will be enough to prove that $\h{U_K}$ is a flat $\h{U_1}$-module. By Proposition \ref{Iadicpiadic}, the $I$-adic completion $V$ of $U_1$ is isomorphic to the closure of the image of $U_1$ in $\h{U}$. Thus we have natural maps $\h{U_1}\to V \to \h{U_K}$. We observe that $V$ is $\pi$-adically complete by the proof of \cite[Theorem VIII.5.14]{ZSv2} noting that ideals in $V$ are $I$-adically closed by \cite[Theorem II.2.1.2, Proposition II.2.2.1]{LVO}

We begin by filtering both $\h{U_1}$  and $V$ $\pi$-adically. Notice that $V/\pi V$ is the $I/\pi U_1$-adic completion of $U_1/\pi U_1$ which is flat by \cite[Proposition 10.14]{AMac}. Since $U_1$ is $\pi$-torsion free, $\gr \h{U_1}\cong (U_1/\pi U_1)[t]$. Similarly, since $V$ is isomorphic to a subring of $\h{U}$, it has no $\pi$-torsion, and so $\gr V\cong (V/\pi V)[t]$. Hence $\gr V$ is flat as a $\gr \h{U_1}$-module. Since both $\h{U_1}$ and $V$ are $\pi$-adically complete, \cite[Proposition 1.2]{ST} implies that $V$ is a flat $\h{U_1}$-module.

Next, we again consider the subspace filtration on $U_1$ induced by the $\pi$-adic filtration on $U$. We have $\gr U \cong \overline{U}[t]$, where $t:=\gr \pi$ and $\overline{U} := U / \pi U$ has degree zero. It follows from Lemma \ref{FsStein}(a) that the image of $\gr U_1$ inside $\gr U$ is equal to $\oplus_{j \geq 0} t^j \cdot \overline{F_jU}$, where $\overline{F_jU}$ is the image of $F_jU$ in $\overline{U}$. Since the quotient filtration $\overline{F_jU}$ on $\overline{U}$ is exhaustive, the localisation of this image obtained by inverting $t$ is equal to $\overline{U}[t, t^{-1}]$. Now $V$ is the completion of $U_1$ so 
\[(\gr V)_t = (\gr U_1)_t = \overline{U}[t, t^{-1}] = \gr\h{U_K}\] 
and therefore $\gr \h{U_K}$ is a flat $\gr V$-module. Hence we can again invoke \cite[Proposition 1.2]{ST} to deduce that $\h{U_K}$ is a flat $V$-module.\end{proof}

Let $U$ be a deformable $\R$-algebra. By functoriality of $\pi$-adic completion, the descending chain 
\[ U = U_0 \supset U_1 \supset U_2 \supset \cdots\]
induces an inverse system of $K$-Banach algebras and bounded algebra maps 
\[ \h{U_K} = \h{U_{0,K}} \leftarrow \h{U_{1,K}} \leftarrow \h{U_{2,K}} \leftarrow  \cdots\]
whose inverse limit we denote by 
\[ \wK{U} := \invlim \hnK{U}.\]
The natural maps $\wK{U}\to \hnK{U}$ may be used to construct semi-norms $|\cdot|_n$ on $\wK{U}$ so that the completion of $\wK{U}$ with respect to $|\cdot|_n$ is $\h{U_{n,K}}$. In this way $\wK{U}$ becomes a Fr\'echet algebra. 
\begin{cor} Let $U$ be a deformable $\R$-algebra such that $\gr U$ is commutative and Noetherian. Then
$\wK{U}$ is a two-sided Fr\'echet--Stein algebra.
\end{cor}
\begin{proof} 
Each $U_n$ is a deformable $\R$-algebra with $\gr U_n \cong \gr U$ by \cite[Lemma 3.5]{AW13}, and the first deformation $(U_n)_1$ of $U_n$ is equal to $U_{n+1}$ by Lemma \ref{FsStein}(b). Hence $\hnK{U}$ is a flat ${\widehat{U_{n+1,K}}}$-module on both sides by Theorem \ref{Onestepflat}. Also each $\hnK{U}$ is Noetherian because $\gr U$ is Noetherian.
\end{proof}

\begin{rmk} Essentially all ideas involved in this proof can already be found in \cite{ST}.
\end{rmk}

\begin{proof}[Proof of Theorem \ref{FsStein}]
Fix $n \geq 0$. Each $K$-Banach algebra $\hK{U(\pi^n \L)}$ is Noetherian. Because $\pi^n \L$ is a smooth $(\R,\A)$-Lie algebra, $U(\pi^n \L)$ is a deformable $\R$-algebra with associated graded ring $\Sym(\pi^n \L)$ by \cite[Theorem 3.1]{Rinehart}, and 
\[U(\pi^n\L)_1 \cong U(\pi^{n+1} \L).\]
Therefore $\hK{U(\pi^n \L)}$ is a flat $\hK{U(\pi^{n+1} \L)}$-module on both sides by the Theorem, and $\w{U(L)} = \invlim \hK{U(\pi^n \L)}$ is two-sided Fr\'echet-Stein.
\end{proof}

\section{The functor \ts{\w\otimes}}\label{Coad}
From now on we will work with categories of \emph{left} modules, however all our results will have analogues valid for categories of right modules. We omit giving the necessary repetitive details to save space.

\subsection{Co-admissible completion}\label{Coadcomplete}

Suppose that $U$ is a left Fr\'echet--Stein algebra. Recall, \cite[\S 3]{ST}, that if $U=\invlim U_n$ is a presentation of $U$ as a left Fr\'echet--Stein algebra then a \emph{coherent sheaf of $U_\bullet$-modules} is a family $(M_n)$ of finitely generated $U_n$-modules $M_n$ together with isomorphisms $U_n\otimes_{U_{n+1}}M_{n+1}\stackrel{\cong}{\longrightarrow} M_n$ for each $n$. The coherent sheaves of $U_\bullet$-modules form an abelian category $\Coh(U_\bullet)$ with respect to the obvious notion of morphism. Then a $U$-module $M$ is said to be \emph{co-admissible} if it is isomorphic as a $U$-module to $\invlim M_n$ for some coherent sheaf of $U_\bullet$-modules $(M_n)$ . By \cite[Lemma 3.8]{ST} the question of whether a $U$-module is co-admissible does not depend of the choice of $U_{\bullet}$ presenting $U$. The co-admissible $U$-modules form a full subcategory $\C_U$ of all $U$-modules. By \cite[Corollary 3.3]{ST} the natural functors 
\[\Gamma\colon \Coh(U_\bullet)\to \C_U \qmb{and} \Loc_{U_\bullet}\colon \C_U\to \Coh(U_\bullet)\]
that send a coherent sheaf $(M_n)$ of $U_\bullet$-modules to the co-admissible $U$-module $\invlim M_n$, and a co-admissible $U$-module $M$ to the coherent sheaf $(U_n\otimes_UM)$ of $U_{\bullet}$-modules, are mutually inverse equivalences of categories. 

\begin{defn} We say that a co-admissible $U$-module $\w{M}$ is a \emph{co-admissible completion} of a $U$-module $M$ if there is a $U$-linear map $\iota_M\colon M\to \w{M}$ such that for every co-admissible $U$-module $N$ and every $U$-linear map $f\colon M\to N$ there is a unique $U$-linear map $g\colon \w{M}\to N$ such that $g\circ \iota_M=f$. 
\end{defn}

By usual arguments with universal properties, if a $U$-module $M$ has a co-admissible completion it (together with the map $\iota$) is uniquely determined up to unique isomorphism.

\begin{prop} Suppose that $U=\invlim U_n$ is a presentation of $U$ as a left Fr\'echet--Stein algebra. If $M$ is a $U$-module such that each $U_n\otimes_U M$ is finitely generated as a $U_n$-module then $\invlim U_n\otimes_U M$ (together with the natural map $\iota_M\colon M\to \invlim U_n\otimes_U M$) is a co-admissible completion of $M$.
\end{prop}

\begin{proof} Certainly $\invlim U_n\otimes_U M$ is a co-admissible $U$-module, so suppose that $N$ is also a co-admissible $U$-module and $f\colon M\to N$ is a $U$-linear map. By functoriality, there is a natural commutative diagram
\[\xymatrix{ M \ar[r]^(.3){\iota_M} \ar_f[d] & \invlim U_n \otimes_U M  \ar[d]^{\w{f}} \\
N \ar[r]_(.3){\iota_N} & \invlim U_n \otimes_U N. }\]
where $\w{f} = \invlim 1\otimes f$. Since $N$ is co-admissible, $\iota_N$ is an isomorphism so we may define $g := \iota_N^{-1} \circ \w{f}$. Then $g \circ \iota_M = f$. 

Suppose that $h : \invlim U_n \otimes_U M \to N$ is another $U$-linear map such that $h \circ \iota_M = f$. Then $\iota_N \circ h \circ \iota_M  = \iota_N \circ f =  \w{f} \circ \iota_M$, so the $U$-linear map 
\[q := \iota_N \circ h - \w{f} : \invlim U_n \otimes_U M \to \invlim U_n \otimes_U N\]
is zero on the image of $\iota_M$. By \cite[Corollary 3.3]{ST}, $q$ is the inverse limit of $U_n$-linear maps $q_n : U_n \otimes_U M \to U_n \otimes_U N$ where $q_n(x_n) = q(x)_n$ for any $x = (x_n) \in \invlim U_n \otimes_U M$. Now for any $m\in M$, $\iota_M(m) = (1 \otimes m) \in \invlim U_n \otimes_U M$, so $q_n(1 \otimes m) = q(\iota_M(m))_n = 0$. Since $q_n$ is $U_n$-linear, we see that $q_n = 0$ for all $n$ and hence $q = 0$. So $\iota_N \circ h = \w{f}$ and $h = \iota_N^{-1}\circ \w{f} = g$.
\end{proof}

\subsection{A Fr\'echet structure on \ts{\Hom} sets for co-admissible modules}\label{FrechetHom}

Suppose that $U$ is a left Fr\'echet--Stein algebra. Let $\Q_U$ be the partially ordered set of continuous seminorms $q$ on $U$ such that the corresponding Banach completion $U_q$  is a left Noetherian $K$-algebra.

\begin{lem} Let $U_\bullet$ be a Fr\'echet--Stein structure on $U$. \be
\item For each $q\in \Q_U$ and $M,N\in \C_U$, $\Hom_{U_q}(U_q\otimes_U M,U_q\otimes_U N)$ is naturally a $K$-Banach space.
\item There is a natural bifunctor from $\Q_U\times \C_U$ to the category of $K$-Banach spaces and continuous maps sending the pair $(q,M)$ to $U_q\otimes_U M$. 
\item For each $M$ and $N$ in $\C_U$, \begin{eqnarray*} \Hom_U(M,N) & \cong & \invlim_{q\in \Q_U} \Hom_{U_q}(U_q\otimes_U M,U_q\otimes_U N)\\ & \cong & \invlim_n \Hom_{U_n}(U_n\otimes_U M,U_n\otimes_U N).\end{eqnarray*} 
\ee
\end{lem}

\begin{proof} Write $q_n$ for the semi-norm on $U$ such that $U_n=U_{q_n}$. For each $q\in \Q_U$, there is some $n$ such that $q_n\leq q$; i.e. the set $\{q_n\}$ is cofinal in $\Q_U$.

(a) Suppose $M, N$ are co-admissible $U$-modules and $q\in \Q_U$. We can find $n$ such that there is a continuous homomorphism of Noetherian $K$-Banach algebras $U_n\to U_q$. Since $M$ and $N$ are co-admissible $U_n\otimes_U M$ and $U_n\otimes_U N$ are finitely generated $U_n$-modules. Thus $U_q\otimes_U M\cong U_q\otimes_{U_n}U_n\otimes_U M$ and $U_q\otimes_U N$ are finitely generated $U_q$-modules. Thus by \cite[Proposition 2.1]{ST}, $U_q\otimes_U M$ and $U_q\otimes_U N$ have canonical Banach topologies and $\Hom_{U_q}(U_q\otimes_U M,U_q\otimes_U N)$ consists of continuous $K$-linear maps. In particular $\Hom_{U_q}(U_q\otimes_U M, U_q\otimes_U N)$ is a closed subspace of the Banach space consisting of all continuous $K$-linear maps from $U_q\otimes_U M$ to $U_q\otimes_U N$. 

(b) Suppose now that $q\leq q'\in \Q_U$ and $M\in \C_U$. Then we can define \[ \psi_{M,q,q'}\colon U_q\otimes_U M\to U_{q'}\otimes_U M\] by identifying $U_{q'}\otimes_U M$ with $U_{q'}\otimes_{U_q}U_q\otimes_U M$ and setting $\psi_{M,q,q'}(u_q\otimes m)=1\otimes u_q\otimes m$.  It is now easy to verify that if $q,q'\in \Q_U$, $M,N\in \C_U$ and $f\in \Hom_U(M,N)$ then  \[\xymatrix{ U_q\otimes_U M \ar[r]^{\id\otimes f} \ar[d]_{\psi_{M,q,q'}} &  U_q \otimes_U N  \ar[d]^{\psi_{N,q,q'}} \\
U_{q'}\otimes_U M \ar[r]_{\id\otimes f} & U_{q'} \otimes_U N }\] commutes. 

(c) Write $M_n=U_n\otimes_U M$ and $N_n=U_n\otimes_U N$. Since the set $q_n$ is cofinal in $\Q_U$, it suffices to show that $\Hom_U(M,N)\cong \invlim_n \Hom_{U_n}(M_n,N_n)$.  Now by the equivalence of categories between coherent $U_\bullet$-modules and coadmissible $U$-modules there is a $K$-linear isomorphism  $\Hom_U(M,N)\cong \Hom_{\Coh U_\bullet}(M_{\bullet},N_{\bullet})$. Thus it remains to observe that if $(f_n)\in \prod_{n\ge 0} \Hom_{U_n}(M_n,N_n)$ then $f_\bullet$ is a morphism of coherent $U_\bullet$-modules if and only if $\psi_{N,q_{n+1},q_n}\circ f_{n+1}=f_n\circ \psi_{M,q_{n+1},q_n}$ for each $n\geq 0$.  
\end{proof}

\begin{defn} Suppose that $M$ and $N$ are co-admissible $U$-modules. Using the lemma we can make \[ \Hom_U(M,N)\cong \invlim_{q\in \Q_U} \Hom_{U_q}(U_q\otimes_U M,U_q\otimes_U N)\] into a $K$-Fr\'echet space by giving it the inverse limit topology in the category of locally convex vector spaces.
\end{defn}

\subsection{The functor \ts{M \mapsto P \w\otimes_V M}}\label{Coadbase}

Let $U$ and $V$ be left Fr\'echet--Stein algebras. 

\begin{defn} We say that a Fr\'echet space $P$ is a \emph{$U$-co-admissible $(U, V)$-bimodule} if $P$ is a co-admissible left $U$-module equipped with a continuous homomorphism $V^{\op}\to \End_U(P)$ with respect to the topology on $\End_U(P)$ defined in \S\ref{FrechetHom}. 
\end{defn}

For any Fr\'echet-Stein structures $U_\bullet$ and $V_\bullet$ on $U$ and $V$ respectively, the definition of the Fr\'echet topology on $\End_U(P)$ implies that $V^{\op} \to \End_U(P)$ is continuous if and only if for every $n \geq 0$, there is some $m\geq 0$ and a continuous algebra homomorphism $V_m^{\op} \to \End_{U_n}(U_n \otimes_U P)$ such that the diagram
\[\xymatrix{ V^{\op}\ar[r]\ar[d] & \End_U(P) \ar[d] \\ V_m^{\op} \ar[r] & \End_{U_n}(U_n \otimes_U P)}\]
commutes. Thus for example $U$ is a $U$-co-admissible $(U, V)$-bimodule whenever $V\to U$ is a continuous homomorphism of left Fr\'echet--Stein algebras.

 \begin{lem} Suppose that $P$ is a $U$-co-admissible $(U,V)$-bimodule. Then for every co-admissible $V$-module $M$, there is a co-admissible $U$-module 
 \[P\w{\otimes}_V M\]
 and a $V$-balanced $U$-linear map 
 \[\iota\colon P\times M\to P\w{\otimes}_V M\]
satisfying the following universal property:  if $f\colon P\times M\to N$ is a $V$-balanced $U$-linear map with $N\in \C_U$ then there is a unique $U$-linear map $g\colon P\w\otimes_V M\to N$ such that $g\iota=f$. Moreover, $P\w\otimes_V M$ is determined by its universal property up to canonical isomorphism.
\end{lem} 

\begin{proof} Let $U=\invlim U_n$ and $V=\invlim V_n$ be presentations of $U$ and $V$ as left Fr\'echet--Stein algebras and let $n \geq 0$ be fixed. Then $P_n := U_n \otimes_U P$ is a ($U_n$, $V$)-bimodule that is finitely generated as a $U_n$-module. Because $V^{\op}\to \End_U(P)$ is continuous, the map $V^{\op}\to \End_{U_n}(P_n)$ factors through $V_m$ for some $m$. Thus
\[P_n\otimes_V M\cong P_n\otimes_{V_m}(V_m\otimes_V M)\]
is a finitely generated $U_n$-module because $M$ is co-admissible. Therefore $P \otimes_V M$ has a co-admissible completion by Proposition \ref{Coadcomplete}, and we define
\[ P \w{\otimes}_V M := \w{P \otimes_V M} = \invlim P_n \otimes_V M.\]
The universal properties of $\otimes_V$ and of co-admissible completion ensure that $P \w{\otimes}_V M$ satisfies the required universal property.
\end{proof}

We note that if $U$, $V$ and $P$ are as in the Lemma then for any choice of $U_\bullet$ presenting $U$ as a left Fr\'echet--Stein algebra, and any $M\in\C_U$, we have isomorphisms 
\[ U_n\otimes_U  (P\w\otimes_V M)\cong U_n\otimes_U P\otimes_V M.\] 
For any $f\in \Hom_{\C_V}(M,M')$, the universal property for $\w\otimes_V$ uniquely determines an element $1\w\otimes f\in \Hom_{\C_U}(P\w\otimes_V M,P\w\otimes_V M')$ since the composite $P\times M\stackrel{1\times f}\longrightarrow P\times M'\to  P\w\otimes_V M'$ is $V$-balanced and $U$-linear. Thus we have defined the \emph{co-admissible base change} functor 
\[P\w\otimes_V- : \C_V\longrightarrow \C_U.\]

\subsection{Associativity of \ts{\w{\otimes}}}\label{WotimesAssoc}

\begin{lem} Suppose that $U$, $V$ and $W$ are left Noetherian $K$-Banach algebras, $P$ is a $(U,V)$-bimodule, and $Q$ is a $(V,W)$-bimodule. Suppose further that $P$ and $Q$ are finitely generated over $U$ and $V$ respectively, and that $V^{\op}\to \End_U(P)$ and $W^{\op}\to \End_V(Q)$ are both continuous. Then $P\otimes_V Q$ is a finitely generated left $U$-module and the natural map $W^{\op}\to \End_U(P\otimes_V Q)$ is continuous.
\end{lem}

\begin{proof} Suppose that $X:=\{x_1,\ldots,x_n\}$ generates $P$ as a left $U$-module and $Y:=\{y_1,\ldots,y_m\}$ generates $Q$ as a left $V$-module. Then if $p\otimes q\in P\otimes Q$, we can write \[ p\otimes q=\sum_{i=1}^m p\otimes v_iy_i=\sum_{i=1}^m pv_i\otimes y_i \] for some $v_1,\ldots,v_m\in V$. Now for each $i$, $pv_i=\sum_{j=1}^n u_{ij} x_j$ for some $u_{ij}\in U$. Thus $p\otimes q=\sum_{i,j} u_{ij}x_j\otimes y_i$. Since $P\otimes_V Q$ is generated by elementary tensors as an  abelian group, it follows that it is  generated as a $U$-module by the set $X\otimes Y:=\{x\otimes y\st x\in X, y\in Y\}$. 

Choose sub-multiplicative norms on $U$, $V$ and $W$ that define their Banach topologies and let $\U$, $\V$ and $\W$ be the corresponding unit balls. By a non-commutative version of \cite[\S 3.7]{BGR}, $\U X$, $\V Y$ and $\U(X\otimes Y)$ are unit balls with respect to norms on $P$, $Q$ and $P\otimes_V Q$ defining their respective Banach topologies. 

Since $V^{\op}\to \End_U(P)$ and $W^{\op}\to \End_V(Q)$ are continuous, there are natural numbers $a$ and $b$ such that $\U X\V\subseteq \pi^{-a}\U X$ and $\V Y\W\subseteq \pi^{-b}\V Y$. Thus 
\[ \U(X\otimes Y)\W \subseteq \U(X \otimes \pi^{-b} \V \Y) = \pi^{-b} \U X \V \otimes Y \subseteq \pi^{-(a+b)} \U(X \otimes Y)\]
and so $W^{\op}\to \End_U(P\otimes_V Q)$ is continuous as claimed.
\end{proof}
\begin{prop} Suppose that $U$, $V$ and $W$ are left Fr\'echet--Stein algebras, that $P$ is a $U$-co-admissible $(U,V)$-bimodule and that $Q$ is a $V$-co-admissible $(V,W)$-bimodule. Then $P \w\otimes_V Q$ is a $U$-co-admissible $(U,W)$-bimodule, and for every co-admissible $W$-module $M$ there is a canonical isomorphism \[ P\w\otimes_V (Q\w\otimes_W M)\stackrel{\cong}{\longrightarrow} (P\w\otimes_V Q)\w\otimes_W M  \] of co-admissible $U$-modules.
\end{prop}

\begin{proof} Let $U_\bullet$, $V_\bullet$ and $W_\bullet$ be Fr\'echet--Stein structures on $U$, $V$ and $W$ respectively. $P\w\otimes_V Q$ is a coadmissible $U$-module by Lemma \ref{Coadbase}, and to see that $W^{\op}\to \End_U(P\w\otimes_V Q)$ is continuous, it suffices to show that for each $n\geq 0$, $W^{\op}\to \End_{U_n}\left(U_n\otimes_U (P\w\otimes_V Q)\right)$ factors continuously through some $W_l^{\op}$.

Fix $n\geq 0$ and write $P_n:=U_n\otimes_U P$. Because $V^{\op}\to \End_U(P)$ is continuous, there is some $m\geq 0$ such that $V^{\op} \to \End_{U_n}(P_n)$ factors through a continuous map $V_m^{\op}\to \End_{U_n}(P_n)$. Let $Q_m:=V_m\otimes_V Q$ so that there is a canonical isomorphism of $(U_n,W)$-bimodules $P_n\otimes_{V_m} Q_m\cong P_n\otimes_V Q$. Because $W^{\op} \to \End_V(Q)$ is continuous, there is some $l\ge 0$ such that $W^{\op}\to \End_{V_m}(Q_m)$ factors through a continuous map $W_l^{\op} \to \End_{V_m}(Q_m)$. Hence $W_l^{\op}\to \End_{U_n}(P_n\otimes_{V_m}Q_m)$ is continuous by the Lemma. 

Now, for the choice of $m$ above, there are canonical isomorphisms
\begin{eqnarray*}
U_n\otimes \left(P\w\otimes_V (Q\w\otimes_W M)\right) & \cong & P_n\otimes_V (Q\w\otimes_W M) \\
								    & \cong &  P_n\otimes_{V_m}(V_m\otimes_V \left(Q\w\otimes_W M)\right)\\
								  & \cong &  P_n\otimes_{V_m} Q_m\otimes_W M \\
								& \cong & P_n\otimes_V Q\otimes_W M\\
								& \cong  & U_n\otimes_U (P\w\otimes_V Q)\otimes_W M\\
								& \cong & U_n\otimes_U \left( ( P\w\otimes_V Q)\w\otimes_W M\right).
\end{eqnarray*}
We note that the composite isomorphism does not depend on $m$ provided that it is sufficiently large with respect to $n$. Since $n$ is arbitrary and both modules in the statement are co-admissible the result follows.
\end{proof}

\begin{cor} Let $W \to V\to U$ be a sequence of continuous morphisms of left Fr\'echet--Stein algebras. Then there is a canonical isomorphism 
\[U\w\otimes_V (V\w\otimes_W M)\stackrel{\cong}{\longrightarrow} U\w\otimes_W M\]
of $U$-modules, for every co-admissible $W$-module $M$.\end{cor}

\subsection{Co-admissible flatness}\label{CoadFlat}
Let $U$ and $V$ be left Fr\'echet--Stein algebras, and let $P$ be a $U$-co-admissible $(U, V)$-bimodule.
 
\begin{defn} \be
\item $P$ is a \emph{c-flat} right $V$-module if $P\w\otimes_V-$ is exact. 
\item $P$ is a \emph{faithfully c-flat} right $V$-module if in addition $P\w\otimes_V M=0$ only if $M=0$.
\ee
\end{defn}

\begin{prop} 
\be \item The functor $P\w\otimes_V-$ is right exact. 
\item If $U=\invlim U_n$ is a presentation of $U$ as a left Fr\'echet--Stein algebra such that $U_n\otimes_U P$ is a flat right $V$-module for all $n$, then $P$ is c-flat over $V$. 
\item If additionally, for all non-zero $M\in \C_V$ there exists $n$ such that $U_n\otimes_U P\otimes_V M$ is non-zero, then $P$ is a faithfully c-flat right $V$-module. 
\ee
\end{prop}

\begin{proof} Let $P_n=U_n\otimes_U P$ and consider the functor $\Loc_{U_\bullet}\circ(P\w\otimes_V -)\colon \C_V\to \Coh(U_\bullet)$. This is equivalent to the functor $(P_n\otimes_V -)$. Since $\Loc_{U_\bullet}$ is an equivalence of categories it suffices to show that $P_n\otimes_V -$ is always right exact, that it is exact if each $P_n$ is flat over $V$, and that if for all non-zero $M\in \C_V$ there exists $n$ such that $P_n\otimes_V M$ is non-zero then $(P_n\otimes_V M)$ is non-zero. All these statements are well-known or clear.
\end{proof}

\subsection{Rescaling the Lie lattice}\label{RescLie}

Suppose that $Y$ is an affinoid subdomain of the reduced $K$-affinoid variety $X$, $\A$ is an affine formal model in $\O(X)$ and $\L$ is a coherent $(\R,\A)$-Lie algebra.

\begin{lem} 
\be \item For each $g\in \O(X)$, there is an $n\ge 0$ such that $\pi^n \L\cdot g\subset \A$.
\item There exists $l\ge 0$ such that $Y$ is $\pi^n\L$-admissible for all $n\ge l$.\ee
\end{lem}
\begin{proof} Suppose that  $x_1,\ldots,x_d$ is a generating set for $\L$ as an $\A$-module. 

(a) Since $x_i\cdot g\in \O(X)$ for each $1\le i\le d$ and $\O(X)=K\cdot \A$, there are $n_i\ge 0$ such that for each such $i$, $\pi^{n_i}x_i\cdot g\in \A$. Taking $n=\sup \{n_i\}$ we see that $\pi^n\L\cdot f\subset \A$ as required.

(b) By Lemma \ref{AdmRalg} and \cite[Proposition 6.2.2.1]{BGR}, there is an affine formal model $\B$ in $\O(Y)$ containing the image of $\A$. Since $\B$ is topologically finitely generated and the action of each $x_i$ on $\B$ is bounded we can find $m_i\ge 0$ such that for each $1\le i\le d$, $\pi^{m_i}x_i\cdot \B\subset \B$. Taking $l=\sup m_i$ we see that $\B$ is $\pi^n\L$-stable for all $n\ge l$. 
\end{proof}

\begin{prop} There is an $m\ge 0$ such that $Y$ is $\pi^n\L$-accessible for all $n\ge m$.
\end{prop} 
\begin{proof} 
First suppose that $Y$ is a rational subdomain of $X$. By \cite[Proposition 7.2.4/1]{BGR}, there is a chain \[Y= Z_{r}\subset Z_{r-1}\subset \cdots\subset Z_1 = X\] such that $Z_{k+1}=Z_k(g_k)$ or $Z_{k+1}=Z_k(1/g_k)$ for some $g_k\in \O(Z_k)$. By part (a) of the Lemma, we may then inductively find $m_k\ge m_{k-1}$ (with $m_0=0$) and $\pi^{m_k}\L$-stable affine formal models $\B_k$ in $\O(Z_k)$ such that $\pi^{m_k}\L\cdot g_k\subset \B_k$. Then $Y\subset X$ is $\pi^n\L$-accessible for all $n\ge m_r$.

Returning to the general case, let $l$ be given by part (b) of the Lemma. By \cite[Theorem 4.10.4]{FvdPut}, we can find rational subdomains $X_1,\ldots,X_r$ of $X$ such that $Y =\bigcup_{j=1}^r X_j$. By part (a), there are $m_1,\ldots,m_r\ge l$ such that $X_j\to X$ is $\pi^n\L$-accessible for $n\ge m_j$ and so we may take $m=\sup m_j$.\end{proof}

\subsection{Theorem}\label{cflat}
Let $Y$ be an affinoid subdomain of the reduced $K$-affinoid variety $X$, let $A = \O(X)$, $B = \O(Y)$ and let $L$ be a coherent $(K,A)$-Lie algebra. Suppose that $L$ has a smooth $\A$-Lie lattice $\L$ for some affine formal model $\A$ in $A$. Then $\w{U(B \otimes_A L)}$ is a co-admissibly flat $\w{U(L)}$-module on both sides.
\begin{proof}
Replacing $\L$ by a $\pi$-power multiple if necessary, by Proposition \ref{RescLie} we may assume that $Y$ is a $\pi^n\L$-accessible affinoid subdomain of $X$ for all $n\ge 0$. Choose an $\L$-stable affine formal model $\B$ in $B$; then $\L' := \B\otimes_{\A}\L$ is a smooth $\B$-Lie lattice in $B \otimes_A L$ so using Definition \ref{DefnOfwUL}, we may write
\[\w{U(L)}=\invlim \hK{U(\pi^{n}\L)}\qmb{and}\w{U(B\otimes_A L)}=\invlim \hK{U(\pi^{n}\L')}.\] 
Now $\w{U(L)}$ is a two-sided Fr\'echet-Stein algebra by Theorem \ref{FsStein}, so $\hK{U(\pi^n \L)}$ is a flat $\w{U(L)}$-module on both sides by the two-sided version of \cite[Remark 3.2]{ST}. Also $\hK{U(\pi^n \L')}$ is a flat $\hK{U(\pi^n \L)}$-module on both sides by Theorem \ref{AffLaccessCovers}(a). Therefore $\hK{U(\pi^n \L')}$ is a flat $\w{U(L)}$-module on both sides, and hence $\w{U(B\otimes_AL)}$ is a co-admissibly flat $\w{U(L)}$-module on both sides by the two-sided version of Proposition \ref{CoadFlat}(b).  
\end{proof} 

\section{Co-admissible \ts{\w{\sU(L)}}-modules on affinoids} \label{CoadsUL}
In this section we suppose that $X$ is a reduced $K$-affinoid variety, $\A$ is an affine formal model in $\O(X)$, $\L$ is a smooth $(\R,\A)$-Lie algebra, and $L = \L \otimes_{\R} K$.

\subsection{Sheaves of Fr\'echet--Stein enveloping algebras}\label{Envaff}
\begin{defn} For each affinoid subdomain $Y$ of $X$, write
\[\w{\sU(L)}(Y) := \w{U(\O(Y) \otimes_{\O(X)} L)}\]
for the Fr\'echet completion of the enveloping algebra $U(\O(Y) \otimes_{\O(X)} L)$. 
\end{defn}

\begin{thm} $\w{\sU(L)}$ is a sheaf of two-sided Fr\'echet--Stein algebras on $X_w$.
\end{thm}
\begin{proof} Let $Y$ be an affinoid subdomain of $X$. By replacing $\L$ by a $\pi$-power multiple if necessary and applying Lemma \ref{RescLie}(b), we may assume that  $Y$ is $\L$-admissible. Let $\B$ be an $\L$-stable affine formal model in $Y$. Then  $\B \otimes_{\A} \L$ is a smooth $\B$-Lie lattice in $\O(Y) \otimes_{\O(X)} L$, so $\w{\sU(L)}(Y)$ is a two-sided Fr\'echet-Stein algebra by Theorem \ref{FsStein}.

By Proposition \ref{FuncFS}(a), $\w{\sU(L)}$ is a presheaf on $X_w$. Let $\U$ be an $X_w$-covering of $X$. By replacing $\L$ by a $\pi$-power multiple again if necessary and applying Lemma \ref{RescLie}(b), we may assume that $\U$ is $\pi^n \L$-admissible for all $n \geq 0$. Now
\[ \w{\sU(L)}(Y) \cong \invlim_{n\geq 0} \hK{\sU(\pi^n \L)} (Y)\]
whenever $Y$ is an intersection of members of $\U$, and the complex $C^\bullet_{\aug}(\U, \hsULnK)$ is exact for each $n\geq 0$ by Corollary \ref{TateFQ}. Therefore
\[C^\bullet_{\aug}(\U, \w{\sU(L)}) \cong \invlim C^\bullet_{\aug}(\U, \hsULnK) \]
is also exact, and hence $\w{\sU(L)}$ is a sheaf. 
\end{proof}

\subsection{Localisation}\label{Loc}
For every co-admissible $\w{U(L)}$-module $M$, we can define a presheaf $\Loc(M)$ of $\w{\sU(L)}$-modules on $X_w$ by setting
\[ \Loc(M)(Y):=\w{\sU(L)}(Y) \underset{\w{U(L)}}{ \w{\otimes}  } M\]
for each affinoid subdomain $Y$ of $X$. The restriction maps in $\Loc(M)$ are obtained from the associativity isomorphism 
\[ \w{\sU(L)}(Z) \underset{\w{\sU(L)}(Y)}{\w\otimes}  \left(\w{\sU(L)}(Y) \underset{\w{U(L)}}{ \w{\otimes}  } M \right) \quad \cong \quad \w{\sU(L)}(Z) \underset{\w{U(L)}}{ \w{\otimes}  } M\]
given by Corollary \ref{WotimesAssoc}.

\begin{thm} $\Loc$ defines a full exact embedding of abelian categories from the category of co-admissible $\w{U(L)}$-modules to the category of sheaves of $\w{\sU(L)}$-modules with vanishing higher \v{C}ech cohomology groups.
\end{thm}
\begin{proof} First we prove that if $M$ is a co-admissible $\w{U(L)}$-module then any $X_w$-covering $\U$ of an affinoid subdomain $Y$ of $X$ is $\Loc(M)$-acyclic. In particular this will demonstrate that $\Loc(M)$ is a sheaf on $X_w$ with vanishing higher \v{C}ech cohomology groups. 

Using Proposition \ref{RescLie}, we may assume that $\U$ is a $\pi^n \L$-accessible covering of $Y$ for each $n\ge 0$. Write $M=\invlim M_n$ where $M_n:=\hK{U(\pi^n\L)}\otimes_{\w{U(L)}}M$, and consider the sheaves $\M_n:=\Loc(M_n)$ of $\hK{\sU(\pi^n\L)}$-modules on $X_{ac}(\pi^n\L)$. By Proposition \ref{LaccLoc}, the augmented \v{C}ech complexes $C_{\aug}^\bullet(\U,\M_n)$ are exact for each $n\ge 0$. 

Now $\Loc(M)(Y)=\invlim \M_n(Y)$ and $\Loc(M)(U)=\invlim \M_n(U)$ for each $U \in \U$. Moreover, by \cite[Theorem B]{ST}, $\invlim{}^{(j)}\M_n(Y)=0 $ and $\invlim{}^{(j)}\M_n(U)=0$ for each $j>0$ and each $U \in \U$. Consider the exact complex of towers of $\w{\sU(L)}(Y)$-modules
\[C^\bullet_{\aug}( \U, (\M_n) ).\]
An induction starting with the left-most term shows that $\invlim{}^{(j)}$ is zero on the kernel of every differential in this complex, for all $j>0$. Therefore $\invlim C_{\aug}^\bullet(\U,\M_n)$ is exact. But this complex is isomorphic to $C_{\aug}^\bullet(\U,\Loc(M))$.

Now suppose that $f\colon M\to N$ is a morphism of co-admissible $\w{U(L)}$-modules. By the universal property of $\w{\otimes}$, for each $Y$ in $X_w$ there is a unique morphism of $\w{\sU(L)}(Y)$-modules $\Loc(M)(Y)\to \Loc(N)(Y)$ making the diagram  \[\begin{CD}
M @>>> N\\
@VVV @VVV\\
\Loc(M)(Y) @>>> \Loc(M)(X) \end{CD}\] commute. It is now easy to see that $\Loc$ is a full functor as claimed.

Finally, suppose that $0\to M_1\to M_2\to M_3\to 0$ is an exact sequence of co-admissible $\w{U(L)}$-modules. Since $\w{\sU(L)}(Y)$ is a c-flat $\w{\sU(L)}(X)$-module on both sides for each $Y \in X_w$ by Theorem \ref{cflat}, each sequence \[0\to \Loc(M_1)(Y)\to\Loc(M_2)(Y)\to \Loc(M_3)(Y)\to 0 \] is exact. This suffices to see that $\Loc$ is exact.\end{proof}

\subsection{Co-admissible \ts{\w{\sU(L)}}-modules}\label{Ucoad}

\begin{defn} Let $\sM$ be a $\w{\sU(L)}$-module. Given an $X_w$-covering $\U=\{U_1,\ldots,U_n\}$ of $X$, we say that $\sM$ is \emph{$\U$-co-admissible} if for each $1\le i\le n$ there is a co-admissible $\w{\sU(L)}(U_i)$-module $M_i$ such that $\sM|_{U_i}$ is isomorphic to $\Loc(M_i)$ as sheaves of $\w{\sU(L)}|_{U_i}$-modules. We say that $\sM$ is \emph{co-admissible} if there is some $X_w$-covering $\U$ of $X$ such that $\sM$ is $\U$-co-admissible.  
\end{defn}

\begin{prop} Suppose that $\alpha\colon \sM\to \sN$ is a morphism of $\U$-co-admissible $\w{\sU(L)}$-modules for some admissible covering $\U$. Then $\ker \alpha, \coker \alpha$ and $\im \alpha$ are each $\U$-co-admissible.
\end{prop}

\begin{proof} We can compute using Theorem \ref{Loc} that $(\ker \alpha)|_{U_i}\cong \Loc(\ker \alpha(U_i))$, that $(\coker \alpha)|_{U_i}\cong \Loc(\coker \alpha(U_i))$ and that $\im \alpha|_{U_i}=\Loc(\im \alpha(U_i))$.
\end{proof}

\begin{lem} Suppose that $\sM$ is a sheaf of $\w{\sU(L)}$-modules isomorphic to $\Loc(M)$ for some co-admissible $\w{U(L)}$-module $M$. Then the sheaf $\Loc\left(\hK{U(\L)}\otimes_{\w{U(L)}}M\right)$ on $X_{ac}(\L)$ has sections given by $Z\mapsto \hsULK(Z)\otimes_{\w{\sU(L)}(Z)}\sM(Z)$.\end{lem}

\begin{proof} The commutative diagram
\[
\xymatrix{ \w{U(L)} \ar@{=}[r]\ar[d] & \w{\sU(L)}(X)  \ar[r]\ar[d] & \w{\sU(L)}(Z)\ar[d]  \\ \hK{U(\L)} \ar@{=}[r] & \hsULK(X) \ar[r]& \hsULK(Z)  }
\]
induces an isomorphism \[ \hsULK(Z)\underset{\hK{U(\L)}}{\otimes} \hK{U(\L)}\underset{\w{U(L)}}{\otimes}M \quad \cong \quad \hsULK(Z) \underset{\w{\sU(L)}(Z)}{\otimes} \w{\sU(L)}(Z)\underset{\w{U(L)}}{\w{\otimes}}M\]
and the result follows.\end{proof}

\subsection{Kiehl's Theorem}\label{Kiehl}

\begin{thm}  Let $\sM$ be a sheaf of $\w{\sU(L)}$-modules on $X_w$. Then the following are equivalent.
\be
\item  $\sM$ is co-admissible.
\item  $\sM$ is $\U$-co-admissible for all $X_w$-coverings $\U$ of $X$.
\item  $\sM$ is isomorphic to $\Loc(M)$ for some co-admissible $\w{U(L)}$-module $M$.
\ee
\end{thm}

\begin{proof} Note that (c)$\implies$(b) and (b)$\implies$(a) are trivial. We will prove (a)$\implies$(c). 

Suppose that $\U$ is a covering of $X$ by affinoid subdomains such that $\sM$ is $\U$-co-admissible. By \cite[Lemmas 8.2.2/2-4]{BGR}, $\U$ may be refined to a Laurent covering $\mathcal{V}=\{ X(f_1^{\alpha_1},\ldots, f_m^{\alpha_m})\st \alpha_i\in\{\pm 1\}\}$ for some $f_1,\ldots,f_m\in\O(X)$. Certainly $\sM$ is $\mathcal{V}$-co-admissible so we may, without loss of generality, assume that $\U=\V$.  Using Proposition \ref{RescLie}, we may also assume that $\U$ is $\pi^n \L$-accessible for all $n\geq 0$.

In an attempt to improve readability, we write $\cS_n$ for the sheaf $\hK{\sU(\pi^n \L)}$ on $X_{ac}(\pi^n \L)$ and $\cS_\infty$ for the sheaf $\w{\sU(L)}$ on $X_w$, so that $\cS_\infty(X) \cong \invlim \cS_n(X)$ and 
\[\cS_\infty(Y) \cong \invlim \cS_n(Y) \qmb{for all} Y \in \U.\]
Fix $n\geq 0$. Consider the sheafification $\M_n$ of the presheaf $Z\mapsto \cS_n(Z)\otimes_{\cS_\infty(Z)}\sM(Z)$ on $X_{ac}(\pi^n\L)$. Let $Y \in \U$, so that $\sM(Y)$ is a co-admissible $\cS_\infty(Y)$-module, and $\sM_{|{Y_w}}$ is isomorphic to $\Loc(\sM(Y))$ by assumption. By Lemma \ref{Ucoad} applied to $\sM_{|Y_w}$ there are isomorphisms
\[\M_n|_{Y_w}\cong \Loc\left(\cS_n(Y)\otimes_{\cS_\infty(Y)}\sM(Y)\right).\]
Thus applying Theorem \ref{LaccKiehl} there is a finitely generated $\cS_n(X)$-module $M_n$, and an isomorphism $\Loc(M_n) \stackrel{\cong}{\longrightarrow} \M_n$. 

Now $\Loc(M_n)\cong \Loc(\cS_n(X)\otimes_{\cS_{n+1}(X)}M_{n+1})$ since they have the same local sections on $\U$. Thus $M_\infty :=\invlim M_n$ is a co-admissible $\cS_\infty(X)$-module. We will show that $\Loc(M_\infty)$ is isomorphic to our sheaf $\sM$.

Let $\theta_n$ denote the $\cS_\infty(X)$-linear map $M_\infty\to \M_n(X)$ defined by the composite of the natural map $M_\infty\to M_n$ and the global sections of the isomorphism $\Loc(M_n)\to \M_n$. Let $Y \in \U$. Combining the isomorphism
\[ \Loc(M_n)(Y) = \cS_n(Y) \otimes_{\cS_n(X)} M_n \stackrel{\cong}{\longrightarrow} \M_n(Y)\]
together with the canonical isomorphism $M_n \cong \cS_n(X) \otimes_{\cS_\infty(X)} M_\infty$ given by \cite[Corollary 3.1]{ST} produces a compatible family of isomorphisms
\[ \alpha_n(Y) : \cS_n(Y) \otimes_{\cS_\infty(X)} M_\infty \stackrel{\cong}{\longrightarrow}\M_n(Y) \] given by the $\cS_\infty(X)$-balanced map $(s, m)\mapsto s\cdot\theta_n(m)|_Y$.
 
Passing to the limit as $n \to \infty$ gives an isomorphism of $\cS_\infty(Y)$-modules
\[ \alpha(Y)\colon \Loc(M_\infty)(Y) = \cS_\infty(Y) \w{\otimes}_{\cS_\infty(X)} M_\infty \stackrel{\cong}{\longrightarrow} \sM(Y)\] given by the $\cS_\infty(X)$-balanced map $(s,m)\mapsto s\cdot \lim(\theta_n(m)|_Y)$.  Since $\sM_{|Y} \cong \Loc(\sM(Y))$ by assumption, Theorem \ref{Loc} gives an isomorphism 
\[\alpha_Y : \Loc(M_\infty)_{|Y} \stackrel{\cong}{\longrightarrow} \sM_{|Y}\]
of sheaves of $\cS_{\infty}|_Y$-modules whose local sections
\[\alpha_Y(Z) \colon \cS_\infty(Z) \stackrel[\cS_\infty(X)]{}{\w{\otimes}} M_\infty  \to \sM(Z)\]
are given by $ \alpha_Y(Z)(s\w{\otimes} m) = s\cdot \lim(\theta_n(m)|_Y)|_Z$, whenever $Z \subset Y$ is an affinoid subdomain of $X$ contained in $Y$. Because $\lim(\theta_n|_Y)|_Z=\lim \theta_n|_{Z}$, it follows that 
\[\alpha_Y(Y\cap Y')=\alpha_{Y'}(Y\cap Y') \qmb{for every} Y,Y'\in \U.\]
Hence the $\alpha_Y$ patch together to an isomorphism of sheaves $\alpha\colon \Loc(M_\infty)\to \sM$.  
 \end{proof}
 
\section{Sheaves on rigid analytic spaces}\label{Sheaves}

In this section $X$ is a rigid $K$-analytic space.  
\subsection{Lie algebroids}\label{LieAlgebroids}
Let $X_w$ denote the subset of $X_{\rig}$ consisting of the affinoid subdomains of $X$. Since we do not assume that $X$ is separated, $X_w$ is not closed under intersections in $X_{\rig}$ and thus is not a $G$-topology on $X$ in general. However, every admissible open subset in $X_{\rig}$ has an admissible cover by affinoid subdomains of $X$. 

\begin{defn}{\cite[\S 9.2.1]{BGR}} A subset $\B$ of objects of $X_{\rig}$ is a \emph{basis} for the topology if every admissible open has an admissible cover by objects in $\B$. \end{defn}

In particular, $X_w$ is a basis of $X$.

\begin{defn}  If $\B$ is a basis of $X$, a presheaf $F$ on $\B$ is a \emph{sheaf} if for every admissible cover $\{U_i\}$ of $U$ by objects in $\B$ and any choice of admissible covers $\{W_{ijk}\}$ of $U_i\cap U_j$, \[ F(U)\to \prod F(U_i)\rightrightarrows \prod F(W_{ijk}) \] is exact. \end{defn}

\begin{thm} Suppose that $\B\subset X_{\rig}$ is a basis for the topology $X$. The restriction functor induces an equivalence of categories between sheaves on $X_{\rig}$ and sheaves on $\B$.\end{thm}

This is a consequence of the Comparison Lemma \cite[Theorem C.2.2.3]{ElephantII}, but we give a proof in Appendix A for the convenience of the reader.

\begin{prop} There is a coherent sheaf $\T_X$ of $K$-Lie algebras on $X_{\rig}$ with \[ \T_X(U):=\Der_K \O(U)\] for every affinoid subdomain $U$ of $X$. Moreover, for all admissible open subsets $Y$ of $X$, $\T_X(Y)$ acts by derivations on $\O_X(Y)$.
\end{prop}
\begin{proof}

We define the restriction maps $\T_X(U)\to \T_X(V)$ for $V\subset U$ affinoid subdomains in $X$ using Lemma \ref{LiftEtale}. By the uniqueness part of that Lemma this defines a presheaf of $K$-Lie algebras on $X_w$. Let $\{U_i\}$ be an admissible affinoid cover of an affinoid subdomain $U$ of $X$. Then it is routine to check that the sequence \[ 0\to \T_X(U)\to \prod \T_X(U_i)\to \prod \T_X(U_i \cap U_j)\] is exact, so $\T_X$ defines a sheaf of $K$-Lie algebras on $X_w$. By the Theorem, this extends to a sheaf of $K$-Lie algebras on $X_{\rig}$. A similarly routine verification shows that $\T_X(Y)$ acts by derivations on $\O_X(Y)$ whenever $Y$ is an admissible open subset of $X$.
\end{proof}

We call the sheaf $\T_X$ constructed in the Proposition the \emph{tangent sheaf} of $X$.

\begin{defn} A \emph{Lie algebroid} on $X$ is a pair $(\rho, \sL)$ such that
\begin{itemize}
\item $\sL$ is a locally free sheaf of $\O$-modules of finite rank on $X_{\rig}$, 
\item $\sL$ has the structure of a sheaf of $K$-Lie algebras, and
\item $\rho\colon \sL\to \T$ is an $\O$-linear map of sheaves of Lie algebras such that \[ [x,ay]=a[x,y]+\rho(x)(a)y\] whenever $U$ is an admissible open subset of $X$, $x,y\in \sL(U)$ and $a\in \O(U)$.
\end{itemize}
\end{defn}

For example, if $X$ is smooth, then the tangent sheaf $\T_X$ is locally free of finite rank by definition, and thus $(\id_{\T_X},\T_X)$ is a Lie algebroid on $X$ by the Proposition.
\subsection{Lie-Rinehart algebras and Lie algebroids}

If $(\rho,\sL)$ is a Lie algebroid on $X$,  then $(\rho(U),\sL(U))$ is a $(K,\O(U))$-Lie algebra for every admissible open subset $U$ of $X$. Moreover every affinoid subdomain $U$ of $X$, $\sL(U)$ is smooth by \cite[Proposition 4.7.2]{FvdPut}.

\begin{defn} A \emph{morphism $(\rho,\sL)\to (\rho',\sL')$ of Lie algebroids} on $X$ is a morphism of sheaves $\theta\colon \sL\to \sL'$ such that $\theta(U)$ is a morphism of $(K,\O(U))$-Lie algebras for every $U\subset X$ in $X_{\rig}$. \end{defn}

\begin{lem} Let $Y=\Sp(A)$ be a $K$-affinoid variety. The global sections functor $\Gamma(Y,-)$ defines an equivalence of categories between the category of Lie algebroids on $Y$ and the category of smooth $(K,A)$-Lie algebras. \end{lem}

\begin{proof} First, suppose that $(L,\rho)$ is a smooth $(K,A)$-Lie algebra and define $\Loc(L)$ to be the locally free sheaf on $Y_w$ given by $\Loc(L)(U)=\O(U)\otimes_AL$ for $U\subset Y$ affinoid and natural restriction maps. By Corollary \ref{LiftEtale} there is a unique structure of a $(K,\O(U))$-Lie algebra on $\Loc(L)(U)$ with anchor map $\rho(U)$ so that \[ \xymatrix{L\ar[r]^\rho \ar[d] & \Der_K(A) \ar[d]\\ \Loc(L)(U) \ar[r]^{\rho(U)} & \T_Y(U) }\] commutes.   Suppose that $V\subset U$ are affinoid subdomains of $Y$, and consider the diagram \[ \xymatrix{ L \ar[r] \ar[d]^\rho & \Loc(L)(U) \ar[r] \ar[d]^{\rho(U)} & \Loc(L)(V) \ar[d]^{\rho'(V)} \\ \Der_K(A) \ar[r] & \T_Y(U) \ar[r] & \T_Y(V)  } \] where $\rho'(V)$ is the anchor map for the unique $(K,\O(V))$-Lie algebra structure on $\Loc(L)(V)$ making the right-hand square commute. Since the left-hand square also commutes, the outer square must commute and $\rho'(V)=\rho(V)$ by the uniqueness of $\rho(V)$. Thus $\rho\colon \Loc(L)\to \T_Y|_{Y_w}$ is a morphism of sheaves of Lie algebras on $Y_w$. By \cite[Proposition 9.2.3/1]{BGR}, $\Loc(L)$ extends to a Lie algebroid $\Loc(L)$ on $Y$.

Now, suppose that $f\colon L\to L'$ is a morphism of $(K,A)$-Lie algebras. By Corollary \ref{LiftEtale} there is a unique morphism of sheaves on $Y_w$ \[ \Loc(f)\colon \Loc(L)|_{Y_{w}}\to \Loc(L')|_{Y_{w}}\] such that $\Loc(f)(Y)=f$, given by $\Loc(f)(U)=\O(U)\otimes_A f$ for affinoid subdomains $U\subset Y$. By \cite[Proposition 9.2.3/1]{BGR} again, $\Loc(f)$ extends to a morphism of Lie algebroids. Thus $\Loc$ defines a functor inverse to $\Gamma(Y,-)$.
\end{proof}
 
\begin{cor} If $(\rho,\sL)$ is a Lie algebroid on a rigid $K$-analytic space $X$, then for every affinoid subdomain $U$ of $X$, $\sL|_{U} \cong \Loc( \sL(U) )$.
\end{cor}

\subsection{The Fr\'echet completion of \ts{\sU(\sL)}}\label{DefnOfDhat}

We will need to work with a slightly coarser basis for $X_{\rig}$ than $X_w$.

\begin{defn} Let $\sL$ be a Lie algebroid on the reduced rigid $K$-analytic space $X$. We say that $\sL(X)$ \emph{admits a smooth Lie lattice} if there is an affine formal model $\A$ in $\O(Y)$ and a smooth $\A$-Lie lattice $\L$ in $\sL(Y)$. We let $X_w(\sL)$ denote the set of affinoid subdomains $Y$ of $X$ such that $\sL(Y)$ admits a smooth Lie lattice.
\end{defn}

\begin{lem} $X_w(\sL)$ is a basis for $X$.
\end{lem} \begin{proof}Suppose that $Y$ is an affinoid subdomain of $X$ such that $\sL(Y)$ is a free $\O(Y)$-module. Then $\sL(Y)$ has a free $\O(Y)^\circ$-lattice spanned by a generating set for $\sL(Y)$ as an $\O(Y)$-module, and some $\pi$-power multiple of this lattice will be a free $\O(Y)^\circ$-Lie lattice by Lemma \ref{LR}(c). Thus $\sL(Y)$ has a smooth $\O(Y)^\circ$-Lie lattice whenever $\sL(Y)$ is a free $\O(Y)$-module, so $X_w(\sL)$ is a basis for $X$ since $\sL$ is a locally free $\O$-module.
\end{proof}

\begin{thm} Let $X$ be a reduced rigid $K$-analytic space. There is a natural functor $\w{\sU(-)}$ from Lie algebroids on $X$ to sheaves of $K$-algebras on $X_{\rig}$ such that there is a canonical isomorphism \[\w{\sU(\sL)}|_{Y_w}\cong\w{ \sU(\sL(Y)) }\]
for every $Y \in X_w(\sL)$. \end{thm}

\begin{proof} Given a Lie algebroid $\sL$ on $X$, let $\w{\sU(\sL)}$ be the presheaf of $K$-algebras on $X_{w}$ given by 
\[\w{\sU(\sL)}(Y) := \w{U(\sL(Y))}\]
on affinoid subdomains $Y$ of $X$, with restriction maps given by Proposition \ref{FuncFS}(a). 

By Theorem \ref{Envaff}, $\w{\sU(\sL)}$ is a sheaf of $K$-algebras on $X_w(\sL)$. Because $X_w(\sL)$ is a basis for $X$ by the Lemma, $\w{\sU(\sL)}$ extends uniquely to a sheaf of $K$-algebras on $X_{\rig}$ by Theorem \ref{LieAlgebroids}.

Moreover if $\sL\to \sL'$ is a morphism of Lie algebroids on $X$ then Proposition \ref{FuncFS}(b), together with Lemma \ref{ResFF}, gives a morphism of sheaves of $K$-algebras $\w{\sU(\sL)}\to \w{\sU(\sL')}$ on $X_{\rig}$ in a functorial way.  
\end{proof}

\begin{defn} We call the sheaf $\w{\sU(\sL)}$ constructed in the Theorem the \emph{Fr\'echet completion of $\sU(\sL)$}. If $X$ is smooth, $\sL=\T$ and $\rho=1_\T$,  we call\[\w{\D}:=\w{\sU(\T)} \] the \emph{Fr\'echet completion of $\D$}.
\end{defn}

From now on we assume that our rigid $K$-analytic space $X$ is reduced.

\subsection{Co-admissible sheaves of modules} \label{CoadmSheaves}

Let $\sL$ be a Lie algebroid on $X$. By analogy with the definition of coherent sheaves given in \cite[\S II.5]{Hart}, we make the following

\begin{defn} A sheaf of $\w{\sU(\sL)}$-modules $\sM$ on $X_{\rig}$ is \emph{co-admissible} if there is an admissible covering $\{U_i\}$ of $X$ by affinoids in $X_w(\sL)$ such that $\sM|_{U_{i,w}}$ is a co-admissible $\w{\sU(\sL)}|_{U_{i,w}}$-module for all $i$ in the sense of Definition \ref{Ucoad}.
\end{defn}

We record three equivalent ways of thinking about co-admissible modules.

\begin{thm} The following are equivalent for a sheaf $\sM$ of $\w{\sU(\sL)}$-modules on $X_{\rig}$:
\be
\item $\sM$ is co-admissible,
\item $\sM|_{U_w}$ is a co-admissible $\w{\sU(\sL)}|_{U_w}$-module for every $U \in X_w(\sL)$,
\item $\sM(U)$ is a co-admissible $\w{\sU(\sL)}(U)$-module, and the natural map
\[\w{\sU(\sL)}(V) \underset{\w{\sU(\sL)}(U)}{\w\otimes}  \sM(U)  \longrightarrow \sM(V)\]
is an isomorphism whenever $V, U \in X_w(\sL)$ and $V \subset U$.
\ee \end{thm}
\begin{proof} (a) $\Rightarrow$ (b). Let $\{U_i\}$ be an admissible affinoid covering of $X$ such that $\sM|_{U_{i,w}}$ is a co-admissible $\w{\sU(\sL)}|_{U_{i,w}}$-module and $U_i \in X_w(\sL)$ for all $i$. Let $U$ be another object of $X_w(\sL)$; then $\{U \cap U_i\}$ is an admissible cover of $U$. Choose an admissible affinoid covering $\{V_{ij}\}_j$ of $U \cap U_i$ for each $i$; then $\{V_{ij}\}_{i,j}$ is an admissible affinoid covering of $U$ and therefore admits a finite subcovering $\W$, say. Now $\sM|_{U_w}$ is $\W$-co-admissible in the sense of Definition \ref{Ucoad} since each $W \in \W$ is an affinoid subdomain of some $U_i$.

(b) $\Rightarrow$ (c). Let $U \in X_w(\sL)$. By Theorem \ref{Kiehl}, $\sM|_{U_w}$ is isomorphic to $\Loc(M_U)$ for some co-admissible $\w{\sU(\sL)}(U)$-module $M_U$. Applying $\Gamma(U,-)$ shows that $M_U = \Loc(M_U)(U) \cong \sM(U)$, so $\sM(U)$ is a co-admissible $\w{\sU(\sL)}(U)$-module. Hence 
\[\sM(V) \cong \Loc\left(\sM(U)\right)(V) = \w{\sU(\sL)}(V)\underset{\w{\sU(\sL)}(U)}{\w\otimes} \sM(U)\]
for every affinoid subdomain $V$ of $U$.

(c) $\Rightarrow$ (a). Using Lemma \ref{DefnOfDhat}, choose an admissible covering $\{U_i\}$ of $X$ by affinoids in $X_w(\sL)$, and let $M_i := \sM(U_i)$ for each $i$. Then $M_i$ is a co-admissible $\w{\sU(\sL)}(U_i)$-module, and there is a natural isomorphism of sheaves of $\w{\sU(\sL)}|_{U_{i,w}}$-modules 
\[\Loc(M_i) \stackrel{\cong}{\longrightarrow} \sM|_{U_{i,w}}\]
 for each $i$, by assumption. Hence $\sM$ is co-admissible.
\end{proof}

It follows readily from Proposition \ref{Ucoad} that the full subcategory of sheaves of $\w{\sU(\sL)}$-modules on $X_{\rig}$ whose objects are the co-admissible $\w{\sU(\sL)}$-modules is abelian.

\subsection{Two corollaries}\label{TwoCors}
We begin with a more general version of Corollary \ref{MainResults}, which follows immediately from Theorems \ref{Loc} and \ref{Kiehl}.

\begin{thm} Suppose that $\sL$ is a Lie algebroid on a reduced $K$-affinoid variety $X$ such that $\sL(X)$ admits a smooth Lie lattice. Then $\Loc$ is an equivalence of abelian categories
\[
\left\{ 
				\begin{array}{c} 
					co\hspace{-0.1cm}-\hspace{-0.1cm}admissible \\ 
					\w{\sU(\sL)}(X)-\hspace{-0.1cm}modules
				\end{array}
\right\} \cong \left\{
				\begin{array}{c}
				 co\hspace{-0.1cm}-\hspace{-0.1cm}admissible \hspace{0.1cm} sheaves \hspace{0.1cm} of\\ 
				  \w{\sU(\sL)} \hspace{-0.1cm}-\hspace{-0.1cm}modules \hspace{0.1cm} on \hspace{0.1cm} X
				\end{array}
\right\}.
\]
\end{thm}

Given an abelian sheaf $\F$ on $X$, write $H^\bullet(X,\F)$ to denote the sheaf cohomology of $\F$, and let $\breve{H}^{\bullet}(\U,\F)$ denote the \v{C}ech cohomology of $\F$ with respect to the covering $\U$. 

\begin{prop} Suppose that $X$ is separated, $\sM$ is a co-admissible $\w{\sU(\sL)}$-module and $\U$ is any cover of $X$ by affinoids in $X_w(\sL)$. Then \[ H^i(X,\sM)= \breve{H}^i(\U,\sM)\] for all $i\ge 0$. In particular, $H^i(X,\sM)=0$ for $i\geq |\U|$. 
\end{prop}
\begin{proof} Since $X$ is separated, every finite intersection $V$ of elements of $\U$ is affinoid. Thus by Theorem \ref{CoadmSheaves}, $\sM|_V$ is a co-admissible $\w{\sU(\sL)}|_V$-module and so has vanishing \v{C}ech cohomology groups by Theorem \ref{Kiehl} and Theorem \ref{Loc}. Thus the result follows from \cite[\href{http://stacks.math.columbia.edu/tag/03F7}{Lemma 03F7}]{stacks-project}.
\end{proof}
It follows from the Proposition that in the setting of the Theorem, the global sections functor $\Gamma(X,-)$ is an exact quasi-inverse to the localisation functor $\Loc$.
\appendix

\section{}\label{ExtendSheaf}

When $X$ is affinoid, \cite[Proposition 9.2.3/1]{BGR} gives that the restriction functor from sheaves on $X_{\rig}$ to sheaves on $X_w$ is an equivalence of categories. In this appendix we extend this result to bases for general rigid $K$-analytic spaces $X$.
 
\subsection{Lemma}\label{ResFF}Suppose that $\B\subset X_{\rig}$ is a basis. The restriction functor $r$ from sheaves on $X_{\rig}$ to the category of sheaves on $\B$ is full and faithful. 
\begin{proof} Suppose that $\F$ and $\G$ are sheaves on $X_{\rig}$ and $\theta$ is a natural transformation from $r (\F) $ to $r (\G) $. We must show that $\theta$ extends uniquely to a morphism of sheaves $t\colon \F\to \G$. 

Suppose that $U$ is an admissible open subset of $X$ and $\U=\{U_i\}$ is a admissible cover of $U$ by $U_i$ in $\B$. Since $\G$ is a sheaf on $X_{\rig}$, \[ \G(U) \to \prod \G(U_i)\rightrightarrows \prod_{i,j} \G(U_i\cap U_j)\] is exact. For each pair $i,j$, choose an admissible cover $\{W_{ijk}\}$ of $U_i\cap U_j$ by objects in $\B$. Since $\G$ is a sheaf, $\G(U_i\cap U_j) \to \prod_k \G(W_{ijk})$ is a monomorphism for each pair $i,j$ and so $\G(U)$ is also the equaliser of $\prod \G(U_i)\rightrightarrows \prod_{ijk} \G(W_{ijk})$. Thus $\G|_{\B}$ is a sheaf on $\B$.

Since $\theta$ is a natural transformation, the two composites \[\F(U)\to \prod \F(U_i)\to \prod \G(U_i)\rightrightarrows \prod \G(W_{ijk})\] agree. Thus there is a unique $t(\U)\in \Hom(\F(U),\G(U))$ such that \[ \xymatrix{ \F(U) \ar[r] \ar[d]^{t(\U)} & \prod\F(U_i) \ar[d]^{\prod \theta(U_i)}  \\ \G(U) \ar[r] & \prod\G(U_i)} \] commutes.  Next, suppose that $\V=\{V_j\}$ is a refinement of $\U$ with each $V_j$ in $\B$. Then   \[ \xymatrix{ \F(U) \ar[r] \ar[d]^{t(\U)} & \prod\F(U_i) \ar[d]^{\prod \theta(U_i)} \ar[r] & \prod \F(V_j) \ar[d]^{\prod \theta(V_j)}  \\ \G(U) \ar[r] & \prod\G(U_i) \ar[r] & \prod \G(V_j) }\] also commutes, so $t(\U)=t(\V)$. Since any two such covers of $U$ have a common refinement, we see that $t(U):=t(\U)$ does not depend on the choice of cover of $U$. In particular if $U$ is in $\B$, $t(U)=\theta(U)$. 

Now suppose that $V\subset U$ are admissible opens in $X_{\rig}$ with $U$ in $\B$. Let $\{V_i\}$ be an admissible cover of $V$ by objects in $\B$. Consider the diagram \[ \xymatrix{ \F(U)\ar[r] \ar[d]^{\theta(U)} & \F(V)\ar[r] \ar[d]^{t(V)} & \prod \F(V_i)\ar[d]^{\theta(V)} \\ \G(U) \ar[r] & \G(V) \ar[r] & \prod \G(V_i). }\] The outer square commutes because $\theta$ is a natural transformation. The right-hand square commutes by the construction of $t(V)$. Since $\G$ is a sheaf, the bottom rightmost horizontal morphism is a monomorphism so it follows that the left-hand square commutes. 

Finally, consider $V\subset U$ for general admissible opens in $X_{\rig}$. Let $\{U_i\}$ be an admissible cover of $U$ by objects in $\B$ and define $V_i:= V\cap U_i$ so that $\{V_i\}$ is an admissible cover of $V$. Then consider the diagram\[ \xymatrix{ \F(U)\ar[r] \ar[d]^{t(U)} & \prod \F(U_i) \ar[r] \ar[d]^{\prod \theta(U_i)} & \prod \F(V_i) \ar[d]^{\prod t(V_i)} \\ \G(U) \ar[r]&  \prod \G(U_i) \ar[r] & \prod \G(V_i). }\] The left-hand square commutes by construction of $t(U)$. The right-hand square commutes by the previous paragraph since each $U_i$ is in $\B$. Thus the outer square commutes. By repeating the argument used in the case $U$ is in $\B$ we see that $t$ is the unique morphism of sheaves extending $\theta$ as required.
\end{proof}

\subsection{Proposition} Suppose $\B\subset X_{\rig}$ is a basis of $X$. The essential image of the restriction functor from sheaves on $X_{\rig}$ to presheaves on $\B$ consists of the sheaves on $\B$. 
\begin{proof} Suppose that $F$ is a sheaf on $\B$. We will construct a sheaf $\F$ on $X_{\rig}$ whose restriction is naturally isomorphic to $F$. Suppose $U$ is an admissible open subvariety of $X$ and $\U=\{U_i\st i\in I\}$ is an admissible cover of $U$ by objects in $\B$. For each $i,j\in I$ let $\V_{ij}=\{V_{ijk}\}$ denote an admissible cover of $U_i\cap U_j$ by objects in $\B$. Then define $H^0(\U, F)$ to be the equaliser of $\prod_{i\in I} F(U_i)\rightrightarrows \prod F(V_{ijk})$. Since $F$ is a sheaf on $\B$, if $\W_{ij}=\{W_{ijl}\}$ is a refinement of $\V_{ij}$ then $\prod F(V_{ijk})\to \prod F(W_{ijl})$ is a monomorphism. Thus $H^0(\U,F)$ only depends on the choice of cover $\U$ not on the choice of $\V_{ij}$. Note also that, by the definition of a sheaf on $\B$, $H^0(\U,F)=F(U)$ whenever $\U$ is a admissible cover of $U\in \B$.

Now, we can define for any admissible open subset $U$ of $X$
\[\F(U):=\dirlim H^0(\U,F)\] 
where the direct limit is over all covers of $U$ by objects in $\B$. In particular $\F(U)\cong F(U)$ for $U\in \B$. Suppose that $V\subset U$ are admissible open subsets of $X$. If $\U=\{U_i\st i\in I\}$ is an admissible cover of $U$ by objects in $\B$ then $\V=\{U_i\cap V\st i\in I\}$ is an admissible cover of $V$. For each $i$ we can find an admissible cover $\V_i$ of $U_i\cap V$ by objects in $\B$. Then $\bigcup \V_i$ is an admissible cover of $V$ by objects in $\B$. Moreover, the universal property of equalisers defines a map $H^0(\U,F)\to H^0(\bigcup \V_i,F)\to \F(V)$. These patch together using the universal property of direct products to give a morphism $\F(V)\to \F(U)$. It is routine to check that in this way $\F$ defines a presheaf on $X_{\rig}$ whose restriction to $\B$ is naturally isomorphic to $F$. The proof of \cite[Lemma 9.2.2/3]{BGR} shows that $\F$ is in fact a sheaf. 
\end{proof}

\bibliography{references}
\bibliographystyle{plain}
\end{document}